\documentclass{amsart} 

\usepackage[english, french]{babel} 
\usepackage[latin1]{inputenc}
\usepackage{amsmath}
\usepackage{amsthm}
\usepackage{amssymb}
\usepackage{tikz}
\usepackage{array}
\usepackage{xy}
\xyoption{all}
\usepackage{mathtools}
\usepackage{imakeidx} 
\usepackage{appendix}
\usepackage{tikz-cd}
\usepackage{verbatim}
\usepackage{enumitem}
\usepackage{multicol}

\usepackage[backref=page]{hyperref}
\usepackage{cleveref}

\hypersetup{colorlinks=true,linkcolor=blue,anchorcolor=blue,citecolor=blue}

\usepackage{adjustbox}

\usepackage{mathtools}

\newtheorem{thm}{Th\'eor\`eme}[section]
\newtheorem{prop}[thm]{Proposition}
\newtheorem{lemma}[thm]{Lemme}

\theoremstyle{definition}

\theoremstyle{remark}

\newtheorem{rmk}[thm]{Remarque}

\numberwithin{equation}{section}

\newcommand{\cont}{{\rm cont}}
\newcommand{\cd}{{\rm cd}}
\newcommand{\Br}{{\rm Br}}
\newcommand{\Nred}{{\rm Nred}}
\newcommand{\Pic}{{\rm Pic}}
\newcommand{\Q}{\mathbb Q}
\newcommand{\R}{\mathbb R}

\renewcommand{\H}{\mathbb H}
\newcommand{\C}{\mathbb C}
\newcommand{\Z}{\mathbb Z}
\newcommand{\G}{{\bf G}}

\renewcommand{\P}{\mathbb P}
\newcommand{\Spec}{\operatorname{Spec}}

\renewcommand{\O}{\mathcal O}

\newcommand{\mc}[1]{\mathcal{#1}}
\newcommand{\cl}{\overline}
\newcommand{\set}[1]{\{#1\}}
\newcommand{\on}[1]{\operatorname{#1}}

\renewcommand{\lim}{\varprojlim}
\newcommand{\et}{\operatorname{\acute{e}t}}

\newcommand{\oi}{\hskip1mm {\buildrel \simeq \over \rightarrow} \hskip1mm}

\title[Sur l'injectivit\'e de l'application cycle]{Sur l'injectivit\'e de l'application cycle de Jannsen}

\author{Jean-Louis Colliot-Th\'el\`ene et Federico Scavia}  

\address{Universit\'e Paris-Saclay, CNRS, Laboratoire de math\'ematiques d'Orsay, 91405, Orsay, France.}
\email{jean-louis.colliot-thelene@universite-paris-saclay.fr}

\address{Department of Mathematics, University of California, Los Angeles, CA 90095-1555, USA}
\email{scavia@math.ucla.edu}

\date{soumis 7 f\'evrier 2023; corrections, 8 avril 2023}

\makeatletter
\@namedef{subjclassname@2020}{\textup{2020} Mathematics Subject Classification}
\makeatother


\begin{document}
	
	\maketitle
	 
	{\it Dedicato alla memoria di Alberto Collino}

	\begin{abstract} 
		Pour certaines classes de vari\'et\'es alg\'ebriques $X$ sur un corps $k$,
		nous comparons deux  applications cycle sur la torsion du groupe
		de  Chow des cycles de codimension 2. La premi\`ere remonte \`a des travaux de S.~Bloch (1981),
		la seconde est l'application cycle de Jannsen \`a valeurs dans la cohomologie  $\ell$-adique continue.
		On obtient ainsi des conditions suffisantes pour l'injectivit\'e de 
		l'application cycle de Jannsen $CH^2(X) \to H^4_{\cont}(X, \Z_{\ell}(2))$
		sur la torsion $\ell$-primaire.
		
		Par ailleurs, pour $k$  un corps de  fonctions rationnelles en une variable 
		sur un corps de nombres totalement imaginaire,  en utilisant des exemples
		de Sansuc et du premier auteur (1983),
		on donne des exemples de surfaces $X/k$ projectives, lisses, g\'eom\'etriquement rationnelles,
		sans point rationnel, pour lesquelles l'application cycle de Jannsen pour $\ell=2$
		a un \'el\'ement de 2-torsion non nul dans son noyau. Ceci r\'epond \`a des questions soulev\'ees dans un article r\'ecent.
	\end{abstract}
	
	\selectlanguage{english}	
	
	\begin{abstract}
		For specific classes of smooth, projective varieties $X$ over a field $k$, we compare two cycle maps on the torsion subgroup $CH^2(X)_{\text{tors} }$ of the second Chow group. The first one goes back to work of S. Bloch (1981), the second one is Jannsen's cycle map into continuous $\ell$-adic cohomology, whose   injectivity properties have attracted attention in two recent papers. The comparison gives sufficient hypotheses to guarantee injectivity of Jannsen's cycle map $CH^2(X) \to H^4_{\cont}(X, \Z_{\ell}(2))$ on $\ell$-primary torsion.
		
		Using counterexamples to injectivity of the first map due to Sansuc and the first author (1983), we give examples of  smooth, projective, geometrically rational surfaces over a rational function field in one variable over a totally imaginary number field  for which Jannsen's map for $\ell=2$ is not injective on $2$-torsion.   This answers  questions raised in a recent paper.
		
	\end{abstract}
	
	\selectlanguage{french}

	\section{Introduction}
	Dans deux articles r\'ecents, 
	F. Suzuki et le second auteur du pr\'esent article \cite{SS22},  puis
	Th. Alexandrou et S. Schreieder \cite{AS22}, se sont int\'eress\'es
	\`a la question suivante. 
	
	Soient $X$ une vari\'et\'e projective lisse connexe sur un corps  $k$ de type fini sur le corps premier. Soit $CH^{i}(X)$ le groupe de Chow des cycles de codimension $i$ modulo l'\'equivalence rationnelle.	Pour $\ell$ premier distinct de la caract\'eristique de $k$, et $i \geq 1$ entier, l'application cycle $\ell$-adique
	continue de Jannsen \cite{J88}  $$CH^{i}(X) \otimes \Z_{\ell} \to H^{2i}_{\cont}(X, \Z_{\ell}(i))$$
	est-elle injective sur la torsion ?
	
	Cette question est \'equivalente \`a la question : l'application cycle
	$$\on{cl}^i\colon CH^{i}(X) \to H^{2i}_{\cont}(X, \Z_{\ell}(i))$$ est-elle injective sur
	la torsion $\ell$-primaire  $CH^{i}(X)\{\ell\}$  de $CH^{i}(X)$ ?

	Nous renvoyons aux introductions de  \cite{AS22} et \cite{SS22} pour l'historique du sujet.
	
	On ne saurait esp\'erer une r\'eponse affirmative sur un corps quelconque.
	Pour $k$ alg\'ebriquement clos, la r\'eponse est  d\'ej\`a n\'egative,
	pour $i=1$ et $X$ une courbe elliptique. 
	
	Dans \cite{SS22}, sur des corps de type fini sur le corps premier,  F. Suzuki et le second auteur donnent des exemples o\`u l'application $\on{cl}^i$ n'est pas injective dans les cas suivants.
	
	(a) \cite[Thm. 1.3]{SS22}  $X$ 
	quotient,  par l'action d'un $2$-groupe fini,
	d'une intersection compl\`ete de grande dimension sur 
	un corps fini ou un corps de nombres,  $i=3$ et $\ell=2$. Les cycles utilis\'es  proviennent de cycles sur une cl\^{o}ture alg\'ebrique
	d\'etect\'es par une m\'ethode topologique.
	
	(b)  \cite[Thm. 1.4]{SS22}
	$X$ de dimension 4 sur un corps de nombres et $i=3$ et $\ell=2$.
	La vari\'et\'e est produit d'un solide de Kummer et d'une courbe elliptique.
	Les cycles utilis\'es sont d\'etect\'es sur une cl\^{o}ture alg\'ebrique
	par une application cycle $\lambda^{3}_{X}$ de Bloch \cite{Blo79}.

	(c) \cite[Thm. 1.5]{SS22} $X$  quadrique de dimension 3 sur  $k(t)$ corps de fonctions rationnelles d'une variable 
	sur un corps de nombres totalement imaginaire,  $i=2$ et  $\ell=2$.
	Ici $X$ n'a pas de point rationnel.  La d\'emonstration de la non trivialit\'e du cycle dans $CH^2(X)[2]$  utilise les r\'esultats de 
	Karpenko \cite{karpenko1991algebro} sur les quadriques de dimension 3 voisines de Pfister.
	Il y a des analogues avec  $i=2$ et $\ell$ premier, avec $X$ de dimension $\ell^2-1$  
	(les vari\'et\'es ``normes'' de Rost) \cite[Thm. 6.3]{SS22}.
	La d\'emonstration de la non trivialit\'e des classes  dans le groupe de Chow  $CH^{2}(X)\{\ell\}$
	utilise les r\'esultats de  Karpenko et Merkurjev \cite{KM13} sur le motif de Rost. La d\'emonstration de la  trivialit\'e de l'image
	par l'application cycle de Jannsen utilise \cite{KM13}  et
	un cas particulier de la proposition \ref{cdetcdcont} ci-dessous.
	Les cycles utilis\'es ont une classe triviale sur une cl\^{o}ture alg\'ebrique.

	Dans \cite{AS22},  Th. Alexandrou et S. Schreieder d\'efinissent, pour toute vari\'et\'e $X$ sur un corps quelconque et tout entier $i$,
	une application cycle $\lambda^{i}_{X}$ sur le groupe   $CH^{i}(X)\{\ell\}$   qui \'etend l'application cycle de  Bloch \cite{Blo79} pour les vari\'et\'es projectives et lisses sur un corps alg\'ebriquement clos.
	Pour une vari\'et\'e lisse sur un corps quelconque et tout entier $i$, l'existence d'applications   $\lambda^{i}_{X}$ \'etendant celle de Bloch \cite{Blo79}  r\'esulte d\'ej\`a de la m\'ethode de Bloch \cite[Cor. 1 p.  772]{CTSS83}   \cite[\S 3.2, diagramme (3.5)]{CT93}. 
	Pour les vari\'et\'es projectives et lisses sur un corps fini,  pour tout entier $i$, une telle application est
	\'etudi\'ee  dans \cite[Cor. 3 (ii) p. 773; Thm. 2, p. 780]{CTSS83}.
	Dans ce cas, l'application  $\lambda^2_{X}$  est injective  et n'est pas plus fine que l'application cycle de Jannsen  \cite[Thm. 4, p. 787]{CTSS83}. La comparaison avec l'application cycle $\ell$-adique
	utilise la commutativit\'e \cite[Prop. 1 p. 766]{CTSS83} \'etablie pour tout $i$ et tout corps.
	
	Entre autres r\'esultats, Th. Alexandrou et S. Schreieder  \'etablissent  sur tout corps l'injectivit\'e de leur 
	application cycle $\lambda^{i}_{X}$ pour $i=1,2$.   
	Pour tout  $d \geq 3$, tout $i$ avec $3 \leq i \leq d$ et tout $\ell$ premier,
	ils donnent des exemples  de vari\'et\'es projectives et lisses de dimension $d$
	sur des corps $k$ de type fini convenable pour lesquelles leur application $\lambda^{i}_{X}$
	n'est pas injective.
	Ils montrent que leur application cycle  $\lambda^{i}_{X}$ raffine, en g\'en\'eral strictement, l'application cycle de Jannsen.
	En corollaire \cite[Cor. 1.4]{AS22}, avec les m\^{e}mes hypoth\`eses sur $i$,  $d$  et $k$ que ci-dessus,  ils obtiennent
	des exemples de non injectivit\'e de l'application
	cycle de Jannsen sur les cycles de torsion.  Pour tout $d\geq 3$, ceci r\'epond \`a la question 
	1.7 (a) de \cite{SS22} sur le cas $i=d$.
	
	Les exemples de \cite[Thm. 1.2]{AS22}  n'ont pas de point rationnel.
	En partant d'un r\'esultat  de Parimala et Suresh (1995) d'application cycle non injective
	pour une surface sur un corps $p$-adique, les auteurs donnent un exemple  \cite[Thm. 1.5]{AS22} de vari\'et\'e projective et lisse $X$ de dimension 3, sur un corps $k$ de type fini sur le corps premier, avec un point rationnel,
	pour laquelle l'application $\lambda^3_{X}$ n'est pas injective (ici $\ell=2$). Il en est donc ainsi aussi  de l'application cycle de Jannsen.

	\medskip

	Ces r\'esultats  laissaient  ouvertes
	les questions d'injectivit\'e
	des  applications cycle de Jannsen  sur  $CH^{i}(X)\{\ell\}$ pour
	les  vari\'et\'es $X$ de g\'eom\'etrie simple, comme  les vari\'et\'es rationnellement 
	connexes,  et pour les surfaces, lorsque le  corps de base
	est de type fini  de petit degr\'e de transcendance sur le corps premier.

	\medskip

	Dans cet article, nous donnons des exemples de non injectivit\'e
	de l'application cycle de Jannsen
	avec $X$ des {\it surfaces g\'eom\'etriquement rationnelles}
	sans point rationnel sur des corps $k$  de type fini
	de degr\'e de transcendance 1 sur un corps de nombres totalement 
	imaginaire, avec $\ell=2$ (Th\'eor\`eme \ref{exemplevoulu}).
	En prenant le produit avec 
	l'espace projectif $\P^{d-2}_{k}$,
	ceci donne de tels exemples  avec $i=2$ et avec $i=d$ pour tout $d \geq 2$, et avec
	$X$ des vari\'et\'es g\'eom\'etriquement rationnelles de dimension $d$.	Les cycles que nous d\'etectons ont une classe triviale sur une cl\^{o}ture alg\'ebrique du corps de base.
	
	\medskip
	
	Voici le contenu de l'article.  
	
	Au \S  \ref{galoisetcontinu} on donne quelques lemmes sur
	la cohomologie continue et la cohomologie galoisienne.
	
	Au \S  \ref{hypotheses},  nous introduisons une classe (\ref{h4})  de vari\'et\'es  projectives et lisses   qui comprend en particulier les surfaces projectives et lisses g\'eom\'etriquement rationnelles. Nous \'etablissons des r\'esultats g\'en\'eraux sur la cohomologie \'etale continue de ces vari\'et\'es.
	
	Au \S  \ref{codim2BCTR}, pour les vari\'et\'es $X$ dans la classe (\ref{h4}) 	nous donnons
	des rappels sur une application cycle d'origine $K$-th\'eorique  sur le groupe de torsion  $CH^2(X)_{\text{tors}}$ du groupe de Chow 
	des cycles de codimension 2
	\'etudi\'ee dans des travaux de S. Bloch, J.-J. Sansuc, le premier auteur,  
	W. Raskind,  B. Kahn.	
	Nous consid\'erons ensuite une application cycle
	de Jannsen  secondaire sur $CH^2(X)_{\text{tors}}$,  
	\`a valeurs dans la cohomologie  \'etale continue de Jannsen.
	
	Une partie essentielle et d\'elicate de l'article consiste \`a \'etablir que les noyaux
	de ces deux applications cycle sont isomorphes. C'est le th\'eor\`eme \ref{compatible-body},
	\'etabli dans l'appendice (Th\'eor\`eme \ref{compatible}). Sous certaines hypoth\`eses
	suppl\'ementaires portant sur l'existence d'un point rationnel ou sur la dimension
	cohomologique du corps de base,
	les travaux mentionn\'es plus haut avaient  \'etabli l'injectivit\'e de l'application cycle d'origine $K$-th\'eorique.
	On en d\'eduit dans ces cas que l'application cycle de Jannsen est injective.
	
	Au \S \ref{fibconiques}  nous rappelons les exemples de non injectivit\'e de l'application cycle d'origine $K$-th\'eorique obtenus par Sansuc et le premier auteur en 1983. Le th\'eor\`eme  \ref{compatible-body} nous permet
	d'en d\'eduire des exemples de non injectivit\'e pour l'application secondaire de Jannsen. Dans notre contexte,
	de tels exemples ne peuvent \^etre construits que sur des corps de dimension cohomologique exactement 3.
	
	Comme indiqu\'e ci-dessus, l'appendice est consacr\'e \`a la comparaison des noyaux des deux  applications cycle sur les vari\'et\'es de la classe  (\ref{h4}). On y utilise des travaux de Lichtenbaum et de Bruno Kahn.

	\medskip

	Les questions suivantes restent ouvertes.
	Sur $k$ un corps de nombres, l'application cycle $CH^{i}(X) \{\ell\} \to H^{2i}_{\cont}(X,\Z_{\ell}(i))$
	est-elle injective dans les cas suivants :
	
	(a) $X$ est une surface et $i=2$.
	
	(b) $X$ est un solide et $i=2$.
	
	(c) $X$ est un solide et $i=3$.

	\bigskip
	
	Fixons quelques notations.
	
	Soient $k$ un corps,
	$k_{s}$ une cl\^{o}ture s\'eparable de $k$ et $G$ le groupe de Galois absolu de $k$.
	\'Etant donn\'e un module galoisien $M$, on note indiff\'eremment $H^{i}(G,M)$
	ou $H^{i}(k,M)$ les groupes de cohomologie galoisienne.
	
	Une $k$-vari\'et\'e est un $k$-sch\'ema s\'epar\'e de type fini.  
	Pour $X$ une $k$-vari\'et\'e, on note $X^s= X \times_{k}k_{s}$.
	Soit $\ell$ un premier inversible dans $k$.
	Sur toute $k$-vari\'et\'e $X$, pour tout entier $n>0$ premier \`a $p$, on a la suite de Kummer pour
	la topologie \'etale
	\begin{equation}\label{kummer}
		\begin{tikzcd}
			1 \arrow[r] & \mu_{\ell^n} \arrow[r] & \G_{m} \arrow[r,"x \mapsto x^{\ell^n}"] & \arrow[r] \G_{m} & 1. 	
		\end{tikzcd}
	\end{equation}
	Pour tout $j \in \Z$, on a des suites exactes
	de modules galoisiens finis 
	\begin{equation}\label{mun-mum}
		\begin{tikzcd}
			1 \arrow[r] & \mu_{\ell^{n}}^{\otimes j} \arrow[r] & \mu_{\ell^{m+n}}^{\otimes j} \arrow[r,"x \mapsto x^{\ell^n}"] & \mu_{\ell^{m}}^{\otimes j}  \arrow[r] \ & 1. 	
		\end{tikzcd}
	\end{equation}
	Ces suites induisent des suites exactes de faisceaux pour la topologie \'etale sur
	toute $k$-vari\'et\'e $X$, compatibles avec les suites de Kummer, 
	en ce sens qu'on a un diagramme commutatif
	\[
	\begin{tikzcd}
		1 \arrow[r] &  \mu_{\ell^{n+1}} \arrow[r] \arrow[d,"x\mapsto x^\ell"] &  \G_{m} \arrow[r]\arrow[d,"x\mapsto x^\ell"] &  \G_{m} \arrow[r]\arrow[d,equal] &  1 \\
		1 \arrow[r] & \mu_{\ell^n}  \arrow[r] &  \G_{m} \arrow[r] &  \G_{m} \arrow[r] &  1.
	\end{tikzcd}
	\]
	
	Pour un groupe ab\'elien $A$ et un nombre premier $\ell$, on note $A[\ell]$ le
	sous-groupe de $\ell$-torsion et $A\{\ell\}$ le sous-groupe de torsion $\ell$-primaire.
	
	Soient $k$ un corps et $X$ une $k$-vari\'et\'e lisse connexe de dimension $d$.
	On note $CH^{i}(X)$ le groupe de Chow des cycles
	de codimension $i$ modulo l'\'equivalence rationnelle. On note
	$CH_{0}(X)$ le groupe de Chow des z\'ero-cycles de dimension z\'ero.
	Si $X$ est projective, on dispose de l'application degr\'e $CH_{0}(X) \to \Z$
	envoyant un point ferm\'e $P \in X$ de corps r\'esiduel $k(P)$ sur le
	degr\'e $[k(P):k]$. On note alors $A_{0}(X) \subset CH_{0}(X)$ le groupe
	des classes de z\'ero-cycles de degr\'e z\'ero.

	\section{Modules galoisiens de type fini et  cohomologie continue}\label{galoisetcontinu}
	Soient $k$ un  corps et $\ell$ un nombre premier inversible dans $k$.
	Soit $M$ un $G$-module $\Z$-libre de type fini. Soit $T$ le $k$-tore de groupe des cocaract\`eres $M$,
	c'est-\`a-dire que l'on a $M\otimes k_{s}^* \oi T(k_{s}).$
	On a les suites exactes d\'eduites de la suite (\ref{kummer}):
	$$ 0 \to M \otimes \mu_{\ell^n} \to M\otimes k_{s}^* \to  M \otimes  k_{s}^* \to 1$$
	soit encore
	\begin{equation}\label{(*)}
		1 \to T(k_{s})[\ell^n] \to T(k_{s}) \to T(k_{s}) \to 1
	\end{equation}
	qui par cohomologie galoisienne donnent les suites exactes
	avec les fl\`eches qu'on imagine en passant de $\ell^m$ \`a $\ell^n$ avec $m \geq n$.
	Plus pr\'ecis\'ement on obtient une surjection
	$T(k)/\ell^m \to T(k)/\ell^n$ induite par l'identit\'e sur $T(k)$ 
	et une application
	$H^1(k,T)[\ell^m]  \to H^1(k,T)[\ell^n] $ induite par la multiplication par $\ell^{m-n}$.
	Comme $H^1(k,T)$ est d'exposant fini (annul\'e par le degr\'e de l'extension finie
	galoisienne d\'eployant   $T$ ou son groupe de cocaract\`eres),
	cette multiplication 
	est {\it nulle}  si $m$ est assez grand. 
	Ainsi pour $n$ fix\'e et $m$ assez grand, l'image de
	$$H^1(k,  M \otimes \mu_{\ell^{n+m}}) \to H^1(k,  M \otimes \mu_{\ell^n})$$
	co\"{\i}ncide avec l'image de $T(k) \to T(k)/\ell^n  \to H^1(k,  M \otimes \mu_{\ell^n})$.
	
	Par ailleurs, pour $n$ assez grand, $ H^1(k, T)/\ell^n = H^1(k,T)\{\ell\}$
	puisque ce dernier groupe est d'exposant fini. On en d\'eduit que
	les applications
	injectives  $$ H^1(k, T)/\ell^n  \to  H^2(k,T(k_{s})[\ell^n])  $$
	d\'eduites de la suite (\ref{(*)}) par cohomologie galoisienne induisent
	des inclusions compatibles
	$$ H^1(k,T)\{\ell\} \hookrightarrow  H^2(k,T(k_{s})[\ell^n]).$$	
	Rappelons qu'un syst\`eme projectif de groupes ab\'eliens $N_{n}$ satisfait la condition de Mittag-Leffler 
	si pour $n\geq 1$ il existe un entier $m(n) \geq n$    tel que tout $m\geq m(n)$,
	l'image de $N_{m} \to N_{n}$ co\"{\i}ncide avec celle de $N_{m(n)} \to N_{n}$.
	Si un syst\`eme projectif $(N_{n})$ de groupes ab\'eliens satisfait la condition de Mittag-Leffler,
	alors ${\lim}^1 N_{n} =0$.

	\begin{prop}
		\label{NSW}\cite[Thm. (2.7.5)]{NSW} 
		Soit $M$ un $\Z_{\ell}$-module de type fini sans torsion \'equip\'e d'une
		action continue de $G$. 
		Pour $n$ entier, $n \geq 1$, soit $M_{n}= M/\ell^n$. On a $M = \lim \ M_{n}$.
		Pour tout $i\geq 1$, on a la suite exacte
		$$ 0 \to { \lim}^1 H^{i-1}(k, M_{n}) 
		\to H^i_{\cont}(k, M )  \to \lim \  H^i(k, M_{n}) \to 0.$$
		Si le syst\`eme $H^{i-1}(k, M_{n})$ satisfait la condition de Mittag-Leffler,
		alors on a
		$$ H^i_{\cont}(k, M ) \oi \lim \  H^i(k, M_{n}).$$
	\end{prop}
	
	Des calculs ci-dessus, on d\'eduit :
	
	\begin{prop}\label{enoncereseaux}
		Soit $M$ un $G$-r\'eseau, groupe des cocaract\`eres   
		d'un $k$-tore $T$.
		
		(a) Le syst\`eme  projectif $H^1(k,  M \otimes \mu_{\ell^n})$
		satisfait la condition de Mittag-Leffler, et on a  ${\lim}^1 \ H^1(k,  M \otimes \mu_{\ell^n})=0$.
		
		(b) La fl\`eche naturelle $H^2_{\cont}(k, M \otimes \Z_{\ell}) \to \lim \ H^2(k, M/\ell^n) $
		est un isomorphisme.
		
		(c) Pour $n$ assez grand, on a des inclusions compatibles 
		$$ H^1(k,T)\{\ell\} \hookrightarrow  H^2(k,T(k_{s})[\ell^n]).$$
	\end{prop}
	
	Notons ici une cons\'equence de la proposition \ref{NSW}.
	
	\begin{prop}\label{cdetcdcont}    Supposons $\cd(k) \leq N$. Pour tout nombre premier $\ell$ et tout entier $j\in \Z$, on a $H^{N+1}_{\cont}(k, \Z_{\ell}(j))=0$.
	\end{prop}
	\begin{proof} D'apr\`es la proposition \ref{NSW},
		on a les suites exactes
		$$ 0 \to {\lim}^1 H^{N}(k,\mu_{\ell^n}^{\otimes j}) \to H^{N+1}_{\cont}(k, \Z_{\ell}(j)) \to \lim 
		\ H^{N+1}(k, \mu_{\ell^n}^{\otimes j}) \to 0.$$ 
		On a les suites exactes courtes (\ref{mun-mum}).
		Sous l'hypoth\`ese $\cd(k) \leq N$, toutes les applications 
		$H^N(k,  \mu_{\ell^{m+n}}^{\otimes j} )\to  H^N(k, \mu_{\ell^n}^{\otimes j})$ sont    surjectives,
		car $H^{N+1}(k,  \mu_{\ell^{m}}^{\otimes j} )=0.$
		La condition de Mittag-Leffler pour la famille $H^N(k, \mu_{\ell^n}^{\otimes j})$
		est donc satisfaite,  et donc ${\lim}^1 H^{N}(k,\mu_{\ell^n}^{\otimes j})=0$.
		Comme les groupes $H^{N+1}(k, \mu_{\ell^n}^{\otimes j})$ sont nuls, on conclut. 
	\end{proof}

	\section{Cohomologie galoisienne et cohomologie \texorpdfstring{$\ell$}{ell}-adique}\label{hypotheses}
	On suppose d\'esormais ${\rm car}(k) =0$. Dans cet article, 
	on s'int\'eresse aux $k$-vari\'et\'es projectives, lisses, g\'eom\'etriquement connexes
	satisfaisant  certaines des  hypoth\`eses suivantes.
	\begin{equation}\label{h1}\tag{H1}
		\text{$\Pic(X^s)=\on{NS}(X^s)$, ce groupe est sans torsion, et  
			$\Br(X^s)=0$. 
		} 
	\end{equation}
	Comme on a suppos\'e ${\rm car}(k)=0$, l'hypoth\`ese (\ref{h1}) est \'equivalente \`a : 
	\begin{equation}\label{h1'}\tag{H1'}
		\text{$H^{i}(X,\O_{X})=0$  pour $i=1,2$ et $H^i(X^s,\Z_{\ell} )\set{\ell}=0$ pour $i= 2, 3$ et tout premier $\ell$.} 
	\end{equation} 
	Si $X$ satisfait (\ref{h1}), alors $H^1(X^s, \Z_{\ell})=0$ et $H^1(X^s, \Z/\ell^n)=0$ pour tout $n\geq 0$ et tout premier $\ell$.
	
	Si $X$ satisfait (\ref{h1}),  $\Pic(X^s)$ est libre de type fini. On notera $S$ le $k$-tore 
	dont le groupe des cocaract\`eres est $\Pic(X^s)$, i.e.  $\Pic(X^s) \otimes k_s^*=S(k_s)$.
	
	Pour $X$ satisfaisant (\ref{h1}), pour tout entier $n>0$, la suite de Kummer pour
	la topologie \'etale (\ref{kummer}) induit des isomorphismes $G$-\'equivariants 
	$$ \Pic(X^s)\otimes \mu_{\ell^n}^{\otimes j}     \oi H^2(X^s, \mu_{\ell^n}^{\otimes j+1})$$
	pour tout entier $j \in \Z$.
	Soit $M=\Pic(X^s)$. 
	On a donc
	$M \otimes \mu_{\ell^n} \oi H^2(X^s, \mu_{\ell^n}^{\otimes 2})$
	et $\Pic(X^{s})\otimes \Z_{\ell}(1) \oi H^2(X^s, \Z_{\ell}(2))$.
	\begin{equation}\label{h2}\tag{H2}
		\text{$X$ satisfait (\ref{h1}) et $H^3(X^s, \Z_{\ell})=0$ pour tout premier $\ell$.}
	\end{equation} 
	\begin{equation}\label{h3}\tag{H3}
		\text{$X$ satisfait (\ref{h1}) et
			$H^3(X^s, \Z/\ell^n)=0$ pour tout $n\geq 0$ et tout premier $\ell$.}
	\end{equation}	
	L'hypoth\`ese (\ref{h3}) \'equivaut \`a la combinaison de (\ref{h2}) et de l'hypoth\`ese que le $\Z_\ell$-module	$H^4(X^s,\Z_{\ell})$ est sans torsion pour tout premier $\ell$.
	L'hypoth\`ese (\ref{h3}) est aussi  \'equivalente \`a l'hypoth\`ese suivante.
    \begin{equation}\label{h4}\tag{H4}
\begin{split}
	H^{i}(X,\O_{X}) = 0 \text{ pour } i=1,2, H^3(X^s, \mathbb{Q}_{\ell}) = 0 &\text{ et } H^i(X^s, \mathbb{Z}_{\ell})\set{\ell} = 0 \\ &\text{pour tout $i\leq 4$ et tout $\ell$.}
\end{split}
\end{equation}
 
	On a  donc les implications suivantes :
	\[\text{(\ref{h4})} \Longleftrightarrow\
	\text{(\ref{h3})}\Longrightarrow\text{(\ref{h2})}\Longrightarrow\text{(\ref{h1})}\Longleftrightarrow\text{(\ref{h1'})}.\]
	Pour $X$ une $k$-surface projective, lisse,
	g\'eom\'etriquement rationnelle,  toutes ces hypoth\`eses sont satisfaites. 
	
	\begin{prop} \label{MLH4} 
		Soit $X$ une $k$-vari\'et\'e satisfaisant (\ref{h3}).   Notons
		$$H^4 (X, \mu_{\ell^n}^{\otimes 2})^0= \on{Ker} [H^4 (X, \mu_{\ell^n}^{\otimes 2}) \to H^4 (X^s, \mu_{\ell^n}^{\otimes 2})].$$
		
		(a)  La suite spectrale de Hochschild-Serre pour la cohomologie \'etale donne une suite exacte
		\begin{equation} \label{S1}
			H^4 (k, \mu_{\ell^n}^{\otimes 2})  \to H^4 (X, \mu_{\ell^n}^{\otimes 2})^0 \to  
			H^2 (k, H^2(X^s, \mu_{\ell^n}^{\otimes 2})).
		\end{equation}

		(b) Si $X$ poss\`ede un point rationnel ou plus g\'en\'eralement un z\'ero-cycle de degr\'e~1,
		la fl\`eche $H^4 (k, \mu_{\ell^n}^{\otimes 2})  \to H^4 (X, \mu_{\ell^n}^{\otimes 2})^0$
		est injective.
		
		(c) Si $\cd(k) \leq 3$,
		alors la fl\`eche $H^4 (X, \mu_{\ell^n}^{\otimes 2})^0 \to  
		H^2 (k, H^2(X^s, \mu_{\ell^n}^{\otimes 2}))$ est injective.
	\end{prop}
	\begin{proof} C'est clair. 
	\end{proof}

	En cohomologie \'etale continue, on dispose de la suite spectrale de Hochschild-Serre
	(\cite[Theorem 3.3]{J88})
	$$E_{2}^{pq} = H^p_{\cont}(k, H^q(X^s, \Z_{\ell}(2))) \Longrightarrow H^*_{\cont}(X, \Z_{\ell}(2)).$$
	
	\begin{prop}
		\label{MLH3}  
		Soit $X$ une $k$-vari\'et\'e satisfaisant (\ref{h2}).  
		
		Notons   $$H^4_{\cont}(X, \Z_{\ell} (2))^0= Ker [H^4_{\cont}(X, \Z_{\ell}(2)) \to H^4(X^s, \Z_{\ell}(2))].$$
		
		(a)  La suite spectrale de Hochschild-Serre pour la cohomologie continue donne une suite exacte
		\begin{equation} \label{S2} H^4_{\cont}(k, \Z_{\ell}(2)) \to H^4_{\cont}(X, \Z_{\ell} (2))^0 \to  H^2_{\cont}(k, H^2(X^s,\Z_{\ell}(2))).		
		\end{equation}
		
		(b) Si $X$ poss\`ede un point rationnel ou plus g\'en\'eralement un z\'ero-cycle de degr\'e~1,
		alors la fl\`eche $H^4_{\cont}(k, \Z_{\ell}(2)) \to H^4_{\cont}(X, \Z_{\ell} (2))$ est injective.
		
		(c) Si l'on a $\cd(k)\leq 3$, la fl\`eche $H^4_{\cont}(X, \Z_{\ell} (2))^0 \to  H^2_{\cont}(k, H^2(X^s,\Z_{\ell}(2)))$
		 
		est injective.
	\end{prop}
	\begin{proof} Pour (a) et (b), c'est clair. Pour le point (c) il suffit de noter que, sous l'hypoth\`ese $\cd(k)\leq 3$,
		la proposition \ref{cdetcdcont} donne  $H^4_{\cont}(k, \Z_{\ell}(2)) =0$.
	\end{proof}
	
	\begin{prop}\label{MLH2}  
		Soit $X$ une $k$-vari\'et\'e satisfaisant (\ref{h1}).
		
		(a)  La fl\`eche naturelle $$H^2_{\cont}(k, H^2(X^s, \Z_{\ell}(2))) \to 
		\lim  H^2(k,H^2(X^s, \mu_{\ell^n}^{\otimes 2}))$$
		est un isomorphisme.
		
		(b)  On a une inclusion naturelle
		$$ H^1(k,\Pic(X^s)  \otimes k_{s}^*)\{\ell\} \hookrightarrow  H^2_{\cont}(k, H^2(X^s, \Z_{\ell}(2))).$$
	\end{prop}
	
	\begin{proof} Il suffit d'appliquer la proposition  \ref{enoncereseaux}
		au $G$-module $M=\Pic(X^s)$, en tenant compte de
		l'identification $\Pic(X^s) \otimes \Z_{\ell}(1) \oi H^2(X^s, \Z_{\ell}(2))$
		provenant des hypoth\`eses (\ref{h1}).
	\end{proof}

	De mani\`ere g\'en\'erale \cite[Chapitre 4, 2.2.10]{deligne1977cohomologie}, pour toute $k$-vari\'et\'e lisse $X$ et tout entier $i$
	on dispose des applications ``cycle'' en cohomologie \'etale
	$$\on{cl}^i_n\colon CH^{i}(X) \to H^{2i}(X, \mu_{\ell^n}^{\otimes i}).$$

	\begin{lemma}  \label{deg1}
		Soient $k$ un corps parfait   et $X$ une $k$-vari\'et\'e  lisse g\'eom\'etriquement int\`egre
		de dimension $d\geq 1 $.  Soit $n>1$ un entier inversible dans $k$.   
		
		Supposons que $X$ poss\`ede un z\'ero-cycle de degr\'e 1.   Soit $M$ un module galoisien fini. 
		
		(a)  Soient  $i \geq 0$ et  $\eta \in H^{i}(X, M)$ dans l'image de $H^{i}(k,  M)$.
		S'il existe un ouvert non vide tel que la restriction de $\eta$ \`a $H^{i}(U, M)$
		soit nulle, alors $\eta=0$.
		
		(b)  En particulier, pour tout $i \geq 1$, toute classe de $H^{2i}(X,  \mu_{n}^{\otimes i})$
		qui est simultan\'ement dans l'image 
		de l'application cycle $CH^{i}(X) \to H^{2i}(X,  \mu_{n}^{\otimes i})$
		et de $H^{2i}(k,  \mu_{n}^{\otimes i}) \to H^{2i}(X,  \mu_{n}^{\otimes i})$
		est nulle.
	\end{lemma}
	
	\begin{proof}
		Par un lemme
		de d\'eplacement facile et connu, 
		il existe un z\'ero-cycle de degr\'e 1 \`a support dans $U$.
		Par fonctorialit\'e contravariante de la cohomologie,
		et un argument de restriction-corestriction,
		ceci \'etablit (a). L'\'enonc\'e (b)
		est  alors une cons\'equence de la fonctorialit\'e contravariante de l'application cycle
		pour les immersions ouvertes.
	\end{proof}

	\section{Cycles de codimension deux et de torsion}\label{codim2BCTR}
	Soient $k$ un corps de caract\'eristique z\'ero et $X$ une $k$-vari\'et\'e projective lisse g\'eom\'e\-tri\-quement connexe.

	\subsection{Une application cycle secondaire provenant de la \texorpdfstring{$K$}{K}-th\'eorie alg\'ebrique}\label{cycleblctra}
	On a l'\'enonc\'e suivant, qui rassemble des travaux  de Raskind et du premier auteur, et de Bruno Kahn,  faisant suite \`a des travaux de  Bloch, Suslin, Merkurjev et Suslin.
	\begin{thm}\label{bloch-ct-raskind}
		Soit $X$ une vari\'et\'e projective, lisse, g\'eom\'etriquement connexe sur un corps $k$
		de caract\'eristique z\'ero.  
		
		(a) Si l'on a $H^1(X, \O_{X})=0$ et $H^2(X^s,\Z_{\ell})_{\on{tors}}=0$ pour tout premier $\ell$, 
		alors
		le groupe  $H^0(X^s,{\mathcal K}_{2})$ est  uniquement divisible.  
		
		(b) Si l'on a $H^2(X,\O_{X})=0$ et  $H^3(X^s,\Z_{\ell})_{\on{tors}}=0$  
		pour tout premier $\ell$, alors la fl\`eche naturelle
		$$ \Pic(X^s) \otimes k_{s}^* \to H^1(X^s,{\mathcal K}_{2})$$
		a son noyau et son conoyau uniquement divisibles.
		
		(c)    Sous la combinaison des hypoth\`eses de (a) et de (b),  c'est-\`a-dire pour $X$ satisfaisant l'hypoth\`ese (\ref{h1}),
		soit $S$ le $k$-tore de groupe des cocaract\`eres le r\'eseau
		galoisien $\Pic(X^s) \oi \rm{NS}(X^s)$.  On a une suite exacte
		\begin{align*} 
			S(k) \rightarrow \operatorname{Ker}[H^3(k,& \mathbb{Q} / \mathbb{Z}(2)) \rightarrow H^3(k(X), \mathbb{Q} / \mathbb{Z}(2))] \\
			\rightarrow &\operatorname{Ker}[ CH^2(X)
			\rightarrow CH^2(X^s )]   \xrightarrow{\Phi} H^1(k, S).
		\end{align*}
		La compos\'ee des fl\`eches
		$$\Pic(X) \otimes k^* \to (\Pic(X^s) \otimes k_{s}^*)^G = S(k) \to H^3(k, \mathbb{Q} / \mathbb{Z}(2))$$
		est nulle.
		
		(d) Si l'on suppose de plus  $H^3(X^s, \mathbb{Q}_{\ell})=0$ et $H^4(X^s, \mathbb{Z}_{\ell})$ sans torsion pour tout premier $\ell$, c'est-\`a-dire pour $X$ satisfaisant l'hypoth\`ese (\ref{h4}),
		alors on a la suite exacte
		\begin{equation}\label{origin}
			S(k) \rightarrow \operatorname{Ker}[H^{ 3 } (k, \mathbb{Q} / \mathbb{Z}(2)) \rightarrow H^3(k(X), \mathbb{Q} / \mathbb{Z}(2))] \rightarrow   CH^2(X)_{\on{tors}}
			\xrightarrow{\Phi} H^1(k, S).
		\end{equation}		
	\end{thm}

	\begin{proof}
		L'\'enonc\'e (a) est un cas particulier de \cite[Thm. 1.8]{CTR85}.
		L'\'enonc\'e (b) est un cas particulier de \cite[Thm. 2.12]{CTR85}.

		Montrons les \'enonc\'es (c) et (d).
		D'apr\`es \cite[Prop. 3.6]{CTR85}, pour toute $k$-vari\'et\'e $X$  lisse
		g\'eom\'etriquement int\`egre, on a une suite exacte
		$$  H^1(X,{\mathcal K}_{2}) \to H^1(X^s,{\mathcal K}_{2})^G \to H^1(G,K_{2}k_{s}(X)/H^0(X^s,{\mathcal K}_{2})) \hskip3cm$$
		$$ \hskip2cm \to \operatorname{Ker} [ CH^2(X)
		\rightarrow CH^2(X^s )] \to H^1(G, H^1(X^s,{\mathcal K}_{2}))$$

		Pour toute $k$-vari\'et\'e lisse, un th\'eor\`eme de Bruno Kahn \cite[Thm. 3.1, Cor.~2 p. 70]{K93}
		donne
		$$
		H^1(G, K_2(k_s(X) / K_2(k_s))) \simeq \operatorname{Ker}[H^3(k, \mathbb{Q} / \mathbb{Z}(2)) \rightarrow H^3(k(X), \mathbb{Q} / \mathbb{Z}(2))].
		$$
		Ce dernier groupe est d'exposant fini, il est annul\'e par le degr\'e de toute extension finie $L/k$ avec $X(L)\neq \emptyset$.
		
		Le groupe $K_{2}k_{s}$ est uniquement divisible. Sous l'hypoth\`ese de (a), le groupe
		$H^0(X^s,{\mathcal K}_{2})$ est uniquement divisible.  On a donc
		$$H^1(G,K_{2}k_{s}(X)/H^0(X^s,{\mathcal K}_{2}))  \simeq
		H^1(G, K_2(k_s(X)) / K_2(k_s)) \hskip1cm  $$ $$ \hskip4cm \simeq \operatorname{Ker}[H^3(k, \mathbb{Q} / \mathbb{Z}(2)) \rightarrow H^3(k(X), \mathbb{Q} / \mathbb{Z}(2))].
		$$
		On peut donc r\'e\'ecrire la suite exacte ci-dessus sous la forme
		$$ H^1(X,{\mathcal K}_{2}) \to H^1(X^s,{\mathcal K}_{2})^G \to 
		\operatorname{Ker}[H^3(k, \mathbb{Q} / \mathbb{Z}(2)) \rightarrow H^3(k(X), \mathbb{Q} / \mathbb{Z}(2))]
		\hskip1cm
		$$
		$$ \hskip4cm \to \operatorname{Ker} [ CH^2(X)
		\rightarrow CH^2(X^s )] \to H^1(G, H^1(X^s,{\mathcal K}_{2})).$$
		D'apr\`es (b), la fl\`eche naturelle
		$$ \Pic(X^s) \otimes k_{s}^* \to H^1(X^s,{\mathcal K}_{2})$$
		a son noyau et son conoyau uniquement divisible. 
		On a $\Pic(X^s) \oi \rm{NS}(X^s)$ donc $\Pic(X^s) \otimes k_{s}^* \oi \rm{NS}(X^s) \otimes k_{s}^*$
		donc $H^1(k, \Pic(X^s) \otimes k_{s}^*)$ est d'exposant fini.
		On en d\'eduit que l'application
		$$ H^1(k, \Pic(X^s) \otimes k_{s}^*)  \to H^1(k, H^1(X^s,{\mathcal K}_{2}))$$
		est un isomorphisme. 
		 			
			En utilisant le fait que   le groupe $\operatorname{Ker}[H^3(k, \mathbb{Q} / \mathbb{Z}(2)) \rightarrow H^3(k(X), \mathbb{Q} / \mathbb{Z}(2))]$ est d'exposant fini,  on d\'eduit de l'\'enonc\'e (b) et d'un argument simple de cohomologie galoisienne que l'image de
			$$ H^1(X^s,{\mathcal K}_{2})^G \to \operatorname{Ker}[H^3(k, \mathbb{Q} / \mathbb{Z}(2)) \rightarrow H^3(k(X), \mathbb{Q} / \mathbb{Z}(2))] $$
			co\"{\i}ncide avec l'image de 
			\begin{align*}
S(k) = (\Pic&(X^s) \otimes k_{s}^*)^G \to  H^1(X^s,{\mathcal K}_{2})^G \\
&\to \operatorname{Ker}\left[H^3(k, \mathbb{Q} / \mathbb{Z}(2)) \rightarrow H^3(k(X), \mathbb{Q} / \mathbb{Z}(2))\right].
\end{align*}
			On a donc obtenu l'\'enonc\'e (c).
			
			Un th\'eor\`eme de Bloch  
			reposant sur le th\'eor\`eme de Merkurjev-Suslin et les conjectures de Weil donne pour toute $k$-vari\'et\'e projective et lisse une injection (voir \cite[Thm. 4.3 (ii)]{CT93}) :
			$$
			CH^2(X^s)\{\ell\} \hookrightarrow H^3(X^s, \mathbb{Q}_{\ell} / \mathbb{Z}_{\ell}(2)) .
			$$
			Sous les hypoth\`eses de (d),  on a
			$$H^3(X^s, \mathbb{Q}_{\ell} / \mathbb{Z}_{\ell}(2))=0.$$
			Ainsi $CH^2(X^s)\{\ell\}=0,$
			et donc
			$$
			CH^2(X)\{\ell\}= \operatorname{Ker}[CH^2(X)\{\ell\} \rightarrow CH^2(X^s)\{\ell\}] .
			$$
			L'\'enonc\'e (d)  suit alors de l'\'enonc\'e (c).
		\end{proof}
		
		\begin{thm}\label{bloch-ct-raskind-2}
			Supposons que $X$ satisfait (\ref{h4}). Dans chacun des cas :
			
			(a) $X$ poss\`ede un z\'ero-cycle de degr\'e $1$,
			
			(b) $cd(k) \leq 2$,
			
			\noindent
			la fl\`eche
			$$
			\Phi\colon CH^2(X)_{\on{tors}} \rightarrow H^1(k, S)
			$$
			est injective. Si de plus le module galoisien $\Pic(X^s)$ est un facteur direct d'un module de permutation, alors $\mathrm{CH}^2(\mathrm{X})_{\on{tors}}=0$.
		\end{thm} 
		
		\begin{proof}
			On a $\operatorname{Ker}[H^3(k, \mathbb{Q} / \mathbb{Z}(2)) \rightarrow H^3(k(X), \mathbb{Q} / \mathbb{Z}(2))]=0$ sous chacune des deux hypoth\`eses. Le th\'eor\`eme \ref{bloch-ct-raskind} donne le r\'esultat.
		\end{proof}
		
		\begin{rmk}
			Si dans l'hypoth\`ese (\ref{h4}) on omet  les conditions sur la nullit\'e de la torsion des
			groupes $H^{i}(X^s,\Z_{\ell})$, tout en gardant les hypoth\`eses $H^{i}(X,\O_{X})=0$ ($i=1,2)$
			et $H^3(X^s,\Q_{\ell})=0$, c'est \`a dire en supposant que le rang $\rho$ du groupe de N\'eron-Severi g\'eom\'etrique et
			les nombres de Betti rationnels $b_{i}$
			satisfont les \'egalit\'es $b_{1}=0, b_{2}=\rho, b_{3}=0$,  	on obtient des  bornes pour l'exposant de torsion de $CH^2(X) _{\on{tors}}$ faisant intervenir 
			les entiers de torsion pour la cohomologie enti\`ere en degr\'es $2,3,4$, le degr\'e d'une extension finie sur laquelle $X$ acquiert un point rationnel,
			et le degr\'e d'une extension de $k$ d\'eployant le module galoisien $\Pic(X^s)$.
		\end{rmk}

		\begin{rmk}
			Sous des hypoth\`eses plus faibles que (\ref{h4}), Shuji Saito  \cite[Thm. B]{Sai91} (voir aussi
			\cite[Thm. 7.2 et Thm. 7.3]{CT93})
			a \'etabli le th\'eor\`eme  d'injectivit\'e suivant pour les applications cycle en cohomologie \'etale.

			\begin{thm}\label{codim2saito} Soit $k$ un corps de car. z\'ero. Soit $X$ une $k$-vari\'et\'e projective
				et lisse g\'eom\'etriquement int\`egre. Supposons $H^{i}(X,\O_{X})=0$
				pour $i=1,2$.  Supposons que l'on a $b_{3}={\rm dim} H^3(X^s,\Q_{\ell})=0$, ou  que le corps
				est $k$ de type fini sur $\Q$.
				
				Faisons de plus  l'une des hypoth\`eses
				
				(a) $X$  poss\`ede un z\'ero-cycle de degr\'e 1.
				
				(b) $cd(k) \leq 2$.
				
				Alors $CH^2(X)_{\on{tors}}$ est annul\'e par un entier $N>0$,
				et, pour tout entier $n>0$ multiple de $N$,
				l'application cycle en cohomologie \'etale
				$$CH^2(X) \to H^4_{\et}(X,\mu_{n}^{\otimes 2})$$
				est injective sur $CH^2(X)_{\text{tors}}$.
				
				En particulier, l'application cycle de Jannsen  $CH^2(X) \to H^4_{\on{cont}}(X,\Z_{\ell}(2) )$
				est injective sur la torsion $\ell$-primaire.
			\end{thm}

			Le th\'eor\`eme \ref{codim2saito} implique :  Pour $X/k$ satisfaisant $H^{i}(X,\O_{X})=0$
			pour $i=1,2$, si  $k$ est de type fini sur $\Q$ 
			et  $X$  poss\`ede un z\'ero-cycle de degr\'e 1,  le groupe 
			$CH^2(X)_{\on{tors}}$ est fini  (Saito \cite[Thm. D]{Sai91},
			voir aussi \cite[Thm. 7.6]{CT93}).
		\end{rmk}

		\subsection{Une application cycle de Jannsen secondaire}\label{cyclejannsen}
		
		Soit $X$ une $k$-vari\'et\'e projective et lisse satisfaisant (\ref{h4}). D'apr\`es la proposition \ref{MLH3}(a), sous les hypoth\`eses $H^1(X^s, \mathbb{Z}_{\ell})=0$ et $H^3(X^s, \mathbb{Z}_{\ell})=0$, la suite spectrale de Hochschild-Serre pour la cohomologie continue donne une suite exacte
		\begin{equation}\label{*hs} H_{\mathrm{cont}}^4(k, \mathbb{Z}_{\ell}(2)) \rightarrow H_{\mathrm{cont}}^4(X, \mathbb{Z}_{\ell}(2))^0 \rightarrow H_{\mathrm{cont}}^2(k, H^2(X^s, \mathbb{Z}_{\ell}(2))),
		\end{equation}
		o\`u \[H_{\mathrm{cont}}^4(X, \mathbb{Z}_{\ell}(2))^0\coloneqq \operatorname{Ker}[H_{\mathrm{cont}}^4(X, \mathbb{Z}_{\ell}(2)) \rightarrow H^4(X^s, \mathbb{Z}_{\ell}(2))].\]
		Sous l'hypoth\`ese (\ref{h4}), on a $H^1(X^s, \mathbb{Z} / \ell^n)=0$ et $H^3(X^s, \mathbb{Z} / \ell^n)=0$. Par la proposition \ref{MLH4}(a), la suite spectrale de Hochschild-Serre pour la cohomologie \'etale usuelle donne une suite exacte
		\begin{equation}\label{**hs}
			H^4(k, \mu_{\ell^n}^{\otimes 2}) \rightarrow H^4(X, \mu_{\ell^n}^{\otimes 2})^0 \rightarrow H^2(k, H^2(X^s, \mu_{\ell^n}^{\otimes 2})),
		\end{equation}
		o\`u 
		\[H^4(X, \mu_{\ell^n}^{\otimes 2})^0=\operatorname{Ker}[H^4(X, \mu_{\ell^n}^{\otimes 2}) \rightarrow H^4(X^s, \mu_{\ell^n}^{\otimes 2})].\]
  
		Pour $X$ une $k$-vari\'et\'e lisse, Jannsen \cite[Theorem 3.23]{J88} a d\'efini des applications cycle
		$$
		\on{cl}^i\colon CH^i(X) \rightarrow H_{\mathrm{cont}}^{2 i}(X, \mathbb{Z}_{\ell}(i)),
		$$
		compatibles avec les applications cycle usuelles
		$$
		\on{cl}^i_n\colon CH^i(X) \rightarrow H^{2 i}(X, \mu_{\ell^n}^{\otimes i}),
		$$
		Nous nous int\'eressons ici au cas $i=2$. Pour $X / k$ une vari\'et\'e satisfaisant (\ref{h4}), et $z \in CH^2(X)\{\ell\}$, l'image de $z$ dans $H^4(X^s, \mathbb{Z}_{\ell}(2))$ est de torsion donc nulle. On obtient donc des applications
		$$
		\Theta\colon CH^2(X)_{\on{tors}} \xrightarrow{\on{cl}^2} H_{\mathrm{cont}}^4(X, \mathbb{Z}_{\ell}(2))^0 \rightarrow H_{\mathrm{cont}}^2(k, H^2(X^s, \mathbb{Z}_{\ell}(2))) .
		$$
		On peut composer avec la r\'eduction des coefficients
		$$
		H_{\mathrm{cont}}^2(k, H^2(X^s, \mathbb{Z}_{\ell}(2))) \rightarrow H^2(k, H^2(X^s, \mu_{\ell^n}^{\otimes 2})) .
		$$
		Notons $\Theta_n(z)$ l'image de $z$.
		
		\begin{prop}\label{Zell-or-elln}
			Supposons que $X$ satisfait (\ref{h4}). Pour tout $z \in CH^2(X)_{\on{tors}}$, on a $\Theta(z)=0$
			si et seulement si $\Theta_n(z)=0$ pour tout entier $n>0$.
		\end{prop} 
		
		\begin{proof}
			Via les applications naturelles \[H_{\mathrm{cont}}^2(k, H^2(X^s, \mathbb{Z}_{\ell}(2))) \rightarrow H^2(k, H^2(X^s, \mu_{\ell^n}^{\otimes 2})),\] l'application $\Theta$ est compatible avec les applications $\Theta_n$ d\'efinies ci-dessus. Il suffit alors de noter que, d'apr\`es la proposition \ref{MLH2}(a), la fl\`eche
			$$
			H_{\mathrm{cont}}^2(k, H^2(X^s, \mathbb{Z}_{\ell}(2))) \rightarrow \lim H^2(k, H^2(X^s, \mu_{\ell^n}^{\otimes 2}))
			$$
			est un isomorphisme. 
		\end{proof} 
		
		\begin{prop}\label{noyau-cd3}
			Supposons que $X$ satisfait (\ref{h4}). Sous l'une ou l'autre des hypoth\`eses
			
			(a) $X$ poss\`ede un z\'ero-cycle de degr\'e $1$,
			
			(b)  $cd(k) \leq 3$,
			
			\noindent on a :
			
			(i) $\on{Ker}(\on{cl}_n^2)\set{\ell}=\on{Ker}(\Theta_n)$ pour tout $n\geq 1$ et
			
			(ii) $\on{Ker}(\on{cl}^2)\set{\ell}=\on{Ker}(\Theta)\set{\ell}$.
		\end{prop} 
		
		\begin{proof}
			(i) On utilise  la suite exacte (\ref{**hs}).  Dans le cas (a) on utilise le lemme \ref{deg1}.
			Dans le cas (b), on a $H^4(k, \mu_{\ell^n}^{\otimes 2})=0$.
			
			(ii) Dans le cas (a), on utilise un analogue du lemme \ref{deg1}.
			Dans le cas (b), on utilise la proposition \ref{cdetcdcont}, qui assure $H_{ {cont}}^4(k, \mathbb{Z}_{\ell}(2))=0$. La suite exacte (\ref{*hs}) donne alors que l'application $H_{\mathrm{cont}}^4(X, \mathbb{Z}_{\ell}(2))^0 \rightarrow H_{\mathrm{cont}}^2(k, H^2(X^s, \mathbb{Z}_{\ell}(2)))$ est injective.
		\end{proof}

		\subsection{Comparaison des deux applications cycle secondaires}
		La d\'emonstra\-tion du th\'eor\`eme suivant est \'etonnamment d\'elicate. Nous y consacrons l'appendice de cet article.
		
		\begin{thm}[Th\'eor\`eme \ref{compatible}]\label{compatible-body}
			Supposons que $X$ satisfait (\ref{h4}).
			Soit $\Phi$ l'application cycle sur  
			$CH^2(X)_{\on{tors}}$ d\'efinie au paragraphe \ref{cycleblctra}.
			Soit $\Theta$ l'application cycle  sur $CH^2(X)_{\on{tors}}$ d\'efinie au paragraphe \ref{cyclejannsen}.
			Alors il existe un isomorphisme
			$$\on{Ker}(\Theta)\set{\ell} \simeq \on{Ker}(\Phi)\set{\ell}.$$
		\end{thm} 
		
		\begin{rmk}
			Soit $X$ une $k$-vari\'et\'e satisfaisant (\ref{h4}).
			Soit $n$ un entier positif. 
			On a un homomorphisme compos\'e :
			\begin{align*} CH^2(X)\{\ell \}  \to &H^1(k,S) \{\ell\}  \to H^2(k, S[\ell^n])) = \\ &= H^2(k, \on{Pic}(X^s) \otimes \mu_{\ell^n})
				\to H^2(k, H^2(X^s,  \mu_{\ell^n}^{\otimes 2})).
			\end{align*}
			La fl\`eche $CH^2(X)\{\ell \}  \to H^1(k,S) \{\ell\} $ est induite par $\Phi$. 
			La multiplication par $\ell^n$ sur le $k$-tore $S$ d\'efinit une suite exacte
			$$ 1 \to S[\ell^n] \to S  \to S \to 1$$
			qui induit la fl\`eche $H^1(k,S) \{\ell\}  \to H^2(k, S[\ell^n]) $. 
			Pour $n$ assez grand, cette fl\`eche  est injective
			(Proposition \ref{enoncereseaux}(b)).
			La fl\`eche naturelle $\on{Pic}(X^s)/ \ell^n \to H^2(X^s, \mu_{\ell^n})$, qui sous l'hypoth\`ese (\ref{h4}) est un isomorphisme,  induit la derni\`ere fl\`eche. 	
			Nous ne savons pas si la fl\`eche compos\'ee ci-dessus co\"{\i}ncide (au signe pr\`es)
			avec l'application $\Theta_{n}$ restreinte \`a la torsion $\ell$-primaire.
		\end{rmk}

		\begin{thm}\label{inj-presque-final} 
			Supposons que $X$ satisfait (\ref{h4}).  Sous l'une  ou l'autre des hypoth\`eses
			
			(a) $X$ poss\`ede un z\'ero-cycle de degr\'e $1$,
			
			(b)  $cd(k) \leq 3$,
			
			\noindent
			on a  un isomorphisme $\on{Ker}(\Phi)\set{\ell} 
			\simeq\on{Ker}(\on{cl}^2)\set{\ell}$.
		\end{thm}
		
		\begin{proof}
			Cela r\'esulte du th\'eor\`eme \ref{compatible-body}  
			et de la proposition \ref{noyau-cd3}.
		\end{proof} 
		
		\begin{thm}\label{inj-final}
			Supposons que $X$ satisfait (\ref{h4}). Sous l'une ou l'autre  des hypoth\`eses
			
			(a) $X$ poss\`ede un z\'ero-cycle de degr\'e $1$,
			
			(b)  $cd(k) \leq 2$,
			
			\noindent
			on a  $\on{Ker}(\on{cl}^2)\set{\ell}=0$.
			
		\end{thm}
		
		\begin{proof}
			On combine  le th\'eor\`eme \ref{compatible-body} et le th\'eor\`eme \ref{bloch-ct-raskind-2}.
		\end{proof}
		
		\begin{rmk}\label{rempfister}
			Soit $X \subset \mathbb{P}_k^4$ une quadrique anisotrope d'\'equation
			$$
			x^2-a y^2-b z^2+a b t^2-c w^2=0 .
			$$
			Alors la fl\`eche
			$$
			\operatorname{Ker}[H ^ { 3 } (k, \mathbb{Q} / \mathbb{Z}(2)) \rightarrow H^3(k(X), \mathbb{Q} / \mathbb{Z}(2))] \rightarrow CH^2(X)_{\text{tors}}
			$$
			est un isomorphisme
			$$
			\mathbb{Z} / 2 \simeq CH^2(X)_{\text{tors}} .
			$$
			Ceci r\'esulte du th\'eor\`eme \ref{bloch-ct-raskind}. 
			De fait dans ce cas $\operatorname{Pic}(X)=\operatorname{Pic}(X^s)=\mathbb{Z}$ avec action triviale de $G$. Donc $H^1(k, S)=0$, 	
			et l'application $ \operatorname{Pic}(X) \otimes k^* \to (\operatorname{Pic}(X^s)\otimes k_{s}^*)^G$ est un isomorphisme.	
			Par ailleurs
			$$
			\operatorname{Ker}[H ^ { 3 } (k, \mathbb{Q} / \mathbb{Z}(2)) \rightarrow H^3(k(X), \mathbb{Q} / \mathbb{Z}(2))]=\mathbb{Z} / 2
			$$
			engendr\'e par la classe du cup-produit $(a) \cup(b) \cup(c) \in H^3(k, \mathbb{Z} / 2)$, comme \'etabli par Arason \cite{arason1975cohomologische}. Que l'on ait $CH^2(X)_{\text{tors}}=\mathbb{Z}/2$ pour une telle quadrique avait \'et\'e \'etabli par
			Karpenko \cite[Theorem 5.3]{karpenko1991algebro}. C'est la classe qui est utilis\'ee dans le th\'eor\`eme $1.5$ ($=6.3$) de \cite{SS22}.
			
			Sur un corps de nombres totalement imaginaire, on n'a pas de tel exemple, car toute forme quadratique en 5 variables a un z\'ero non trivial.
			
			Sur le corps $k=\mathbb{R}$ ou $k=\mathbb{Q}$, on trouve donc une quadrique $X \subset \mathbb{P}_{\mathbb{R}}^4$ pour laquelle l'application $\Phi\colon CH^2(X)_{\text{tors}} \rightarrow H^1(k, S)$ a un noyau non nul de 2-torsion.
			
			Par contre, O. Wittenberg (communication personnelle) a montr\'e
			que sur $\R$ comme sur un corps de nombres, par exemple $\Q$, pour toute vari\'et\'e $X$	satisfaisant (\ref{h4}), et en particulier pour la quadrique dans $\P^4_k$ sans $k$-point ci-dessus, l'application
			$CH^2(X)_{\on{tors}} \to H^4_{\on{cont}}(X, \Z_2(2))$	est injective.
			Voir aussi \cite[Remark 6.4]{SS22}.
		\end{rmk}

		\section{Surfaces fibr\'ees en coniques sur la droite projective}\label{fibconiques}

		Dans ce paragraphe, par $k$-surface rationnelle on entend
		une $k$-surface projective,  lisse, g\'eom\'etriquement rationnelle. Une telle surface $X$  satisfait (\ref{h4}).
		On note $S$ le $k$-tore de groupe des cocaract\`eres $\Pic(X^s)$.
		La suite exacte (\ref{origin})	
		$$	 S(k) \rightarrow H^1(k, K_2(k_s(X))
		/ K_2(k_s)) \rightarrow \operatorname{Ker} [ CH^2(X)
		\rightarrow CH^2(X^s )] \rightarrow H^1(k, S)
		$$
		se lit ici
		\begin{equation}
			\label{Bl81}
			S(k) \rightarrow  H^1(k, K_2(k_s(X))
			/ K_2(k_s)) \rightarrow A_{0}(X)  \xrightarrow{\Phi}  H^1(k, S).
		\end{equation}
		Elle avait \'et\'e \'etudi\'ee dans ce cas par Bloch \cite{Blo81} puis par Sansuc et le premier auteur \cite{CTS81}. 
		
		Soit $K=k(\P^1)=k(t)$.
		Soit $X\to \P^1_{k}$ une surface projective et  lisse sur $k$, g\'eom\'etriquement connexe,
		fibr\'ee en coniques sur $\P^1_{k}$,
		de fibre g\'en\'erique donn\'ee   par la conique d'\'equation homog\`ene dans $\P^2_{K}$:
		$$U^2-a(t) V^2 -b(t) W^2=0.$$

		\begin{lemma}[\cite{MN}]\label{max-noether}
			La surface $X$ est rationnelle.
		\end{lemma}
		
		\begin{proof}
			Rappelons la d\'emonstration moderne.  Toute conique  sur le corps 
			$k_{s}(t)$ admet un point rationnel (Tsen). Si la conique est lisse,
			elle est isomorphe \`a $\P^1_{k_{s}(t)}$. Son corps des fonctions, qui
			est celui de $X$, est transcendant pur sur $k_{s}$.	
		\end{proof}
		
		Soit $A=(a,b)$ l'alg\`ebre de quaternions associ\'ee sur le corps $K$ et  soit
		$q$ la forme quadratique diagonale $\langle 1,-a,-b, ab \rangle$ sur $K$. 
		
		Le groupe des normes r\'eduites   $\Nred(A^*) \subset K^*$
		est le groupe des \'el\'ements non nuls  de $K^*$ repr\'esent\'es par la forme quadratique $q$.
		Soit  $K^*_{q} \subset K^*$
		le sous-groupe form\'e des \'el\'ements $f \in K^*$ tels que la
		forme quadratique $f.q \perp -q$ sur le corps $K$ soit l'image d'une forme
		quadratique sur $k$. Soit  $k^*_q \subset k^*$ 
		le groupe $k^* \cap K^*_{q} \subset K^*$.
		Compl\'etant des calculs de S. Bloch \cite[\S 3]{Blo81}, Sansuc et le premier auteur ont \'etabli :
		\begin{thm} \cite[Prop. 2, p. 437]{CTS81}.  Avec les notations ci-dessus,
			on a  $$H^1(k, K_{2}(k_{s}(X))/K_{2}(k_{s})) \oi K^*_{q}/\Nred(A^*),$$
			et le quotient  de  $$S(k) \to H^1(k, K_{2}(k_{s}(X))/K_{2}(k_{s}))$$
			dans la suite (\ref{Bl81}) 
			s'identifie \`a  $K^*_{q}/k^*_{q}\cdot\Nred(A^*)$.
			On a  la suite exacte
			$$ 1 \to K^*_{q}/k^*_{q}\cdot\Nred(A^*) \to A_{0}(X)  {\hskip1mm \buildrel \Phi \over \rightarrow \hskip1mm}H^1(k,S).$$
		\end{thm}

		En utilisant ce calcul explicite du noyau de $\Phi$, Sansuc et le premier auteur ont construit dans \cite[\S 3, p. 464-465]{CTS83} 
		des exemples de surfaces fibr\'ees en coniques  sur la droite projective pour lesquelles
		le noyau de $\Phi$
		n'est pas nul. D'apr\`es le th\'eor\`eme \ref{inj-final}, ceci ne peut
		se produire si $X$ a un z\'ero-cycle de degr\'e 1 ou si $\cd(k)\leq 2$.

		\medskip
		
		On a donn\'e de tels exemples\footnote{\`A la page 465, ligne -15 de \cite{CTS83}, il faut lire : il n'existe pas $a \in \Q^*$.} sur tout corps $k$ avec
		$\Q(t) \subset k \subset \Q_{3}(\!(t)\!)$, et sur tout corps $k$ avec
		$\C(x,y,z) \subset k \subset \C(\!(x)\!)(\!(y)\!)(\!(z)\!)$.
		En particulier, pour tout $p \geq  5$, congru \`a $-1$ modulo $3$,   il existe des exemples avec le corps 
		$k=\Q(\sqrt{-p})(t)$,  dont la dimension cohomologique est 3.
		Un exemple concret est fourni pour tout tel  $p$ par
		une surface fibr\'ee en conique sur $\P^1_{k}$ de fibre  g\'en\'erique donn\'ee par
		la conique d'\'equation homog\`ene
		$$ X^2+t Y^2 + 3 (s^2-2)(s^2-3) T^2=0,$$
		sur le corps $K= k(\P^1)= \Q(\sqrt{-p})(t)(s)$.
		
		La forme quaternionienne associ\'ee est 
		$$q= \langle 1, -t, -3 (s^2-2)(s^2-3), 3t (s^2-2)(s^2-3)\rangle.$$
		
		L'\'el\'ement $f= 2(s^2-3)$ est dans $K^*_{q}$ mais pas dans
		$k^*_q\cdot D_{K}(q)$.

		En combinant cela avec le th\'eor\`eme \ref{inj-presque-final}, on obtient :
		
		\begin{thm}\label{exemplevoulu}
			Soit $p$ un nombre premier tel que $p\geq 5$ et $p\equiv -1\pmod 3$, soit
			$k=\Q(\sqrt{-p})(t)$ le corps des fractions rationnelles en une variable sur le corps $\Q(\sqrt{-p})$.
			C'est un corps de dimension cohomologique 3.
			Il existe une surface $X/k$ projective, lisse et g\'eom\'etriquement
			rationnelle pour laquelle le noyau de l'application cycle de Jannsen
			$CH^2(X)  \to H^4_{\cont}(X, \Z_{2}(2))$
			contient un \'el\'ement non nul de 2-torsion.
		\end{thm}

		Ceci r\'epond \`a la question 1.7(a)  de \cite{SS22}.

		\begin{rmk}
			En prenant le produit de $X$ par un espace projectif $\P^{d-2}_{k}$,  
			on obtient en toute
			dimension $d\geq 2$ des exemples de vari\'et\'es $Y$ projectives, lisses,
			g\'eom\'etri\-que\-ment rationnelles sur le corps $k=\Q(\sqrt{-p})(t)$
			pour lesquelles les applications cycle de Jannsen
			ne sont pas injectives sur $CH^2(Y)[2]$ et pour lesquelles elles ne
			sont pas injectives sur $CH^d(Y)[2]$.
		\end{rmk}

		\begin{rmk}
			Dans le th\'eor\`eme \ref{exemplevoulu}, le fait que l'on a $\cd(k)=3$	joue un double r\^ole.
			D'une part le th\'eor\`eme \ref{inj-final} montre qu'on ne saurait
			avoir un exemple de non injectivit\'e avec $\cd(k)\leq 2$. 
			D'autre part, pour assurer la traduction entre l'exemple
			de \cite{CTS83}  concernant l'application $\Phi$ de Bloch
			et le pr\'esent exemple avec l'application de Jannsen en cohomologie continue,
			on a  utilis\'e $\cd(k)\leq3$, par le biais de
			la proposition \ref{inj-presque-final}(b), qui repose sur  la proposition
			\ref{noyau-cd3}, laquelle repose sur la proposition \ref{cdetcdcont}.
		\end{rmk}
		
		\appendix
		
		\section{La deuxi\`eme application d'Abel-Jacobi sup\'erieure}
		
		L'objectif de cette annexe est la preuve de l'assertion suivante.	
		\begin{thm}\label{compatible}
			Soient $k$ un corps de caract\'eristique z\'ero, $\ell$ un nombre premier et $X$ une $k$-vari\'et\'e projective lisse et g\'eom\'etriquement connexe qui satisfait (\ref{h4}). Alors il existe un isomorphisme $\on{Ker}(\Theta)\set{\ell}\simeq \on{Ker}(\Phi)\set{\ell}$.
		\end{thm}
		
		On fixe quelques notations pour cette annexe. Soient $k$ un corps, $G$ le groupe de Galois absolu de $k$,  $X$ une $k$-vari\'et\'e et $p\colon X\to \Spec(k)$ le morphisme structural. On \'ecrit $X_{\et}$ pour le petit site \'etale de $X$. 
		
		Si $F$ est un faisceau ab\'elien sur $X_{\et}$, $p_*F$ est le faisceau associ\'e au $G$-module continu $F(X^s)\coloneqq \varinjlim_{L/k} F(X_L)$, o\`u $L/k$ parcourt l'ensemble des sous-extensions finies de $k_s/k$, et on a $\Gamma(\Spec(k)_{\et}, p_*F)=\Gamma(X_{\et},F)=F(X^s)^G$. Si $A$ est un complexe de faisceaux ab\'eliens sur $X_{\et}$, on notera par $\H^*(X,A)$ l'hypercohomologie de $A$. On notera par $F^s$ (resp.  $A^s$) l'image r\'eciproque de $F$ (resp. $A$) le long du morphisme de projection $X^s\to X$.
		
		\subsection{Homomorphismes dans la suite spectrale de Hochschild-Serre}
		Soit $A$ un complexe de faisceaux ab\'eliens sur $X_{\et}$. On suppose que $A$ appartient \`a $D^{\geq 0}(X_{\et})$, c'est-\`a-dire $\H^i(X,A)=0$ pour tout $i<0$. La suite spectrale de Grothendieck associ\'ee \`a la composition $\Gamma(k,-)\circ p_* = \Gamma(X,-)$ 
		est la suite spectrale de Hochschild-Serre
		\begin{equation}\label{hochschild-serre-hyper}
			E_2^{i,j}\coloneqq H^i(k,\H^j(X^s,A^s))\Rightarrow \H^{i+j}(X,A).
		\end{equation}	
		La suite spectrale (\ref{hochschild-serre-hyper}) est fonctorielle en $A$. On note $F^q\H^*(X,A)$ la filtration d\'ecroissante induite sur $\H^*(X,A)$. Pour tout $r\geq 1$ on a des homomorphismes
		\[\varphi^1\colon F^1\H^r(X,A)\to H^1(k,\H^{r-1}(X^s,A^s)).\]
		Si $\H^{r-1}(X^s,A^s)=0$, alors la diff\'erentielle $d_2^{0,r-1}$ est nulle et on obtient un homomorphisme
		\[\varphi^2\colon F^2\H^r(X,A)\to H^2(k,\H^{r-2}(X^s,A^s)).\]
		On utilisera la proposition suivante (plus pr\'ecis\'ement, la remarque \ref{translation} qui en est cons\'equence) pour construire le carr\'e commutatif (\ref{bottom-right}).
		\begin{prop}\label{bord}
			Soit 
			\begin{equation}\label{abc}
				A\to B\to C\to A[1]
			\end{equation} 
			un triangle exact dans $D^{\geq 0}(X_{\et})$ et soit $\partial\colon \H^*(X,C)\to \H^{*+1}(X,A)$ l'homomorphisme de bord induit par (\ref{abc}). 	Soit $j\geq 2$ un entier. Supposons que l'on a $\H^{j-1}(X^s,A^s)=0$  et que l'application $\H^{j-3}(X^s,B^s)\to \H^{j-3}(X^s,C^s)$ est surjective, ce qui donne une suite exacte courte de $G$-modules
			\begin{equation}\label{abc-2}0 \to \H^{j-2}(X^s,A^s)\to \H^{j-2}(X^s,B^s)\to \H^{j-2}(X^s,C^s) \to 0.
			\end{equation}
			Alors $\partial(\H^{j-1}(X,C))\subset F^2\H^j(X,A)$ et on a un diagramme commutatif 
			\[
			\begin{tikzcd}
				F^1\H^{j-1}(X,C) \arrow[r,"\varphi^1"] \arrow[d,"\partial"]   & H^1(k, \H^{j-2}(X^s,C^s)) \arrow[d,"\delta"]  \\
				F^2\H^j(X,A) \arrow[r,"\varphi^2"] & H^2(k,\H^{j-2}(X^s,A^s))
			\end{tikzcd}
			\]
			o\`u $\delta$ est l'homomorphisme de bord associ\'e \`a (\ref{abc-2}).
		\end{prop}
		
		\begin{proof}	
			Si $D=(D^{i,j})_{i,j\geq 0}$ est un complexe double, o\`u $i$ d\'enote la coordonn\'ee  horizontale et $j$ la coordonn\'ee
			verticale, on notera $\on{Tot}(D)$ le complexe total associ\'e \`a $D$, par $d_0$ (resp. $d_1$) la diff\'erentielle verticale (resp. horizontale) de la deuxi\`eme suite spectrale associ\'ee \`a $D$, et par $d$ la diff\'erentielle totale. On \'ecrira $H_{d_0}(D)$ pour le complexe double obtenu en prenant la cohomologie de $D$ par rapport \`a $d_0$ (c'est la page $E_1$ de la deuxi\`eme suite spectrale associ\'ee \`a $D$). Si $x\in \on{Tot}(D)^j$ satisfait $d(x)=0$, on notera $\cl{x}$ sa classe dans $H^j(\on{Tot}(D))$. Si $x^{i,j}\in D^{i,j}$ satisfait $d_0(x^{i,j})=0$ et $d_1(x^{i,j})\in \on{Im}(d_0)$, on \'ecrira $[x^{i,j}]$ pour sa classe modulo $\on{Im}(d_0)+\on{Im}(d_1)$, c'est-\`a-dire dans la page $E_2$ de la deuxi\`eme suite spectrale associ\'ee \`a $D$.
			
			D'apr\`es le lemme du fer \`a cheval (Horseshoe Lemma en anglais, le dual de \cite[Lemma 2.2.4]{weibel1994introduction}), on peut construire un diagramme commutatif de complexes doubles de $G$-modules 
			\begin{equation}\label{cartan-eilenberg1}
				\begin{tikzcd}
					0 \arrow[r] & A' \arrow[r]\arrow[d] & B' \arrow[r]\arrow[d] & C' \arrow[r]\arrow[d] & 0 \\
					0 \arrow[r] & I \arrow[r,"\iota"] & J \arrow[r,"\pi"] & K \arrow[r] & 0,	
				\end{tikzcd}
			\end{equation}		
			o\`u la ligne du haut est une suite exacte courte de complexes de $G$-modules (vus comme complexes doubles concentr\'es dans la ligne z\'ero) qui repr\'esente le triangle \[Rp_*A\to Rp_*B\to Rp_*C\to Rp_*A[1],\] les fl\`eches verticales sont des
			r\'esolutions de Cartan-Eilenberg, et pour tout $q\geq 0$ la suite exacte de complexes de $G$-modules
			\[0\to I^{*,q}\xrightarrow{\iota^{*,q}} J^{*,q}\xrightarrow{\pi^{*,q}} K^{*,q}\to 0\]
			est scind\'ee.  En particulier, la suite du bas reste exacte apr\`es passage aux $G$-invariants.

			On peut visualiser la r\'esolution $A'\to I$ comme suit:
			\begin{equation}\label{cartan-eilenberg-explicit}
				\begin{tikzcd}
					\vdots & \vdots & \vdots & \vdots	 \\
					(A')^2 \arrow[r] \arrow[u] & I^{0,2} \arrow[u,"d_0"]\arrow[r,"d_1"] & I^{1,2} \arrow[u,"d_0"]\arrow[r,"d_1"] & I^{2,2}\arrow[r,"d_1"]\arrow[u,"d_0"] & \cdots \\
					(A')^1 \arrow[r]\arrow[u] & I^{0,1} \arrow[r,"d_1"]\arrow[u,"d_0"] & I^{1,1}\arrow[r]\arrow[u,"d_0"] & I^{2,1}\arrow[r,"d_1"]\arrow[u,"d_0"] & \cdots\\ 
					(A')^0 \arrow[r]\arrow[u] & I^{0,0} \arrow[r,"d_1"]\arrow[u,"d_0"] & I^{1,0}\arrow[r,"d_1"]\arrow[u,"d_0"] & I^{2,0}\arrow[r,"d_1"]\arrow[u,"d_0"] & \cdots.
				\end{tikzcd}
			\end{equation}
			Une description similaire s'applique \`a $B'\to J$ et $C'\to K$.
			
			Par d\'efinition, la suite spectrale de Hochschild-Serre pour $A$ (resp. $B$, $C$) est la deuxi\`eme suite spectrale du complexe double $I^G$ (resp. $J^G$, $K^G$) et les homomorphismes naturels entre ces suites spectrales sont induits par $\iota$ et $\pi$. La suite spectrale de Hochschild-Serre pour $A$ est donc obtenue par passage aux $G$-invariants dans (\ref{cartan-eilenberg-explicit}) ($A'$ exclu). Sa page $E_1$ est donn\'ee par $H_{d_0}(I^G)$, et sa page $E_2$ est obtenue en prenant la cohomologie de $E_1$ par rapport 
 \`a la diff\'erentielle induite par $d_1$. Les suites spectrales de Hochschild-Serre de $B$ et $C$ admettent des descriptions analogues.

			Le diagramme (\ref{cartan-eilenberg1}) induit un diagramme commutatif de $G$-modules
			\begin{equation}\label{cartan-eilenberg2}
				\begin{tikzcd}
					0 \arrow[r] & \H^{j-2}(X^s,A^s)\arrow[r] \arrow[d]  & \H^{j-2}(X^s,B^s)\arrow[r]\arrow[d] & \H^{j-2}(X^s,C^s) \arrow[r]\arrow[d] & 0\\
					0 \arrow[r] & H_{d_0}(I)^{*,j-2}\arrow[r,"H_{d_0}(\iota)"] & H_{d_0}(J)^{*,j-2}\arrow[r,"H_{d_0}(\pi)"] & H_{d_0}(K)^{*,j-2} \arrow[r] & 0,
				\end{tikzcd}
			\end{equation}
			o\`u les fl\`eches verticales sont des r\'esolutions injectives. Ici, la suite exacte sup\'erieure est (\ref{abc-2}): son exactitude suit de la nullit\'e du groupe $\H^{j-1}(X^s,A^s)$ et de la surjectivit\'e de l'application $\H^{j-3}(X^s,B^s)\to \H^{j-3}(X^s,C^s)$. La suite exacte de complexes de $G$-modules inf\'erieure est scind\'ee ligne par ligne et reste donc exacte apr\`es passage aux $G$-invariants.

			Le fait que $\partial(\H^{j-1}(X,C))\subset F^2\H^j(X,A)$ suit de la fonctorialit\'e de la suite spectrale de Hochschild-Serre pour $A$ et $C[1]$ par rapport au morphisme $A\to C[1]$ et du fait que $F^1\H^j(X,A)=F^2\H^j(X,A)$. 
			
			Tout \'el\'ement de $F^1\H^{j-1}(X,C)$ est la classe $\cl{c}$ d'un \'el\'ement $$c=(c^{i,j-i-1})_{i=0}^{j-1}\in \on{Tot}(K^G)^{j-1}$$ tel que $d(c)=0$ et $c^{0,j-1}=0$. On  en d\'eduit  $d_0(c^{1,j-2})=d_1(c^{0,j-1})=d_1(0)=0$. On a aussi $d_1(c^{1,j-2})=d_0(c^{2,j-3})\in \on{Im}(d_0)$. Par d\'efinition $\varphi^1(\cl{c})=[c^{1,j-2}]$. 
			
			La construction de $\partial(\cl{c})$ se fait \`a l'aide du diagramme suivant induit par (\ref{cartan-eilenberg1}): 
			\[
			\begin{tikzcd}
				& \on{Tot}(J^G)^{j-1}\arrow[r,"\pi"] \arrow[d,"d"] & \on{Tot}(K^G)^{j-1} \\
				\on{Tot}(I^G)^j \arrow[r,"\iota"] & \on{Tot}(J^G)^j
			\end{tikzcd}
			\]
			Plus pr\'ecis\'ement, il existe $b=(b^{i,j-i-1})_{i=0}^{j-1}\in \on{Tot}(J^G)^{j-1}$ tel que $\pi(b)=c$ et $b^{0,j-1}=0$. Alors $\pi(d(b))=d(c)=0$, donc il existe un unique $a=(a^{i,j-i})_{i=0}^{j}\in \on{Tot}(I^G)^j$ tel que $\iota(a)=b$. On a $d(a)=0$ et par d\'efinition $\partial(\cl{c})=\cl{a}$.
			
			Comme $b^{0,j-1}=0$ et $\iota$ est injectif, on a $a^{0,j}=0$. Comme $\cl{a}\in F^2\H^j(X,A)$, il existe $a^{0,j-1}\in (I^G)^{0,j-1}$ et $a^{1,j-2}\in (I^G)^{1,j-2}$ tels que $d_0(a^{0,j-1})=0$ et $a^{1,j-1}=d_1(a^{0,j-1})+d_0(a^{1,j-2})$. Si on remplace $b$ par $b-\iota(a^{0,j-1}+a^{1,j-2})$ et $a$ par $a-d(a^{0,j-1}+a^{1,j-2})$, on a encore $\pi(b)=c$, $\iota(a)=d(b)$ et $\partial(\cl{c})=\cl{a}$, et on a aussi $a^{0,j}=a^{1,j-1}=0$. En particulier, $d_0(a^{2,j-2})=0$. En outre $d_1(a^{2,j-2})=-d_0(a^{3,j-3})\in \on{Im}(d_0)$. Par d\'efinition $\varphi^2(\cl{a})=[a^{2,j-2}]$, donc
			\begin{equation}\label{vers-1} \varphi^2(\partial(\cl{c}))=\varphi^2(\cl{a})=[a^{2,j-2}].
			\end{equation}
			
			L'\'el\'ement $\delta(\varphi^1(\cl{c}))=\delta([c^{1,j-2}])$ se construit \`a partir du diagramme suivant induit par (\ref{cartan-eilenberg2}):
			\[
			\begin{tikzcd}
				& (H_{d_0}(J)^{1,j-2})^G\arrow[r,"H_{d_0}(\pi)"]\arrow[d,"H_{d_0}(d_1)"] & (H_{d_0}(K)^{1,j-2})^G \\
				(H_{d_0}(I)^{2,j-2})^G\arrow[r,"H_{d_0}(\iota)"] & (H_{d_0}(J)^{2,j-2})^G & 
			\end{tikzcd}
			\]
			Plus pr\'ecis\'ement, il existe $\beta^{1,j-2}\in J^{1,j-2}$ tel que $d_0(\beta^{1,j-2})=0$, la classe de $\beta^{1,j-2}$ dans $H_{d_0}(J)^{1,j-2}$ est $G$-invariante et $\pi(\beta^{1,j-2})=c^{1,j-2}+d_0(\gamma^{1,j-3})$ avec $\gamma^{1,j-3}\in K^{1,j-3}$. Comme $\pi$ est surjectif, on peut trouver $\beta^{1,j-3}\in J^{1,j-3}$ tel que $\pi(\beta^{1,j-3})=\gamma^{1,j-3}$. Si on remplace $\beta^{1,j-2}$ par $\beta^{1,j-2}-d_0(\beta^{1,j-3})$, alors $d_0(\beta^{1,j-2})=0$, la classe de $\beta^{1,j-2}$ dans $H_{d_0}(J)^{1,j-2}$ est encore $G$-invariante et $\pi(\beta^{1,j-2})=c^{1,j-2}$. Il existe un unique $\alpha^{2,j-2}\in I^{2,j-2}$ tel que $\iota(\alpha^{2,j-2})=d_1(\beta^{1,j-2})$. Comme $d_0(\beta^{1,j-2})=0$ et la classe de $\beta^{1,j-2}$ dans $H_{d_0}(I)^{2,j-2}$ est $G$-invariante, on a $d_0(\alpha^{2,j-2})=0$ et la classe de $\alpha^{2,j-2}$ dans $H_{d_0}(I)^{2,j-2}$ est $G$-invariante. Par d\'efinition $\delta([c^{1,j-2}])=[\alpha^{2,j-2}]$, donc 
			\begin{equation}\label{vers-2}
				\delta(\varphi^1(\cl{c}))=\delta([c^{1,j-2}])=[\alpha^{2,j-2}].
			\end{equation}
			Comme $\pi(b^{1,j-2})=\pi(\beta^{1,j-2})=c$, il existe un unique $\alpha^{1,j-2}$ tel que $b^{1,j-2}-\beta^{1,j-2}=\iota(\alpha^{1,j-2})$. L'homomorphisme $\iota$ \'etant injectif, par application de $d_1$ on obtient $a^{2,j-2}-\alpha^{2,j-2}=d_1(\alpha^{1,j-2})$, donc
			\begin{equation}\label{vers-3}
				[a^{2,j-2}]=[\alpha^{2,j-2}].
			\end{equation}
			La combinaison de (\ref{vers-1}), (\ref{vers-2}) et (\ref{vers-3}) entra\^ine $\varphi^2(\partial(\cl{c}))=\delta(\varphi^1(\cl{c}))$.
		\end{proof}
		
		\begin{rmk}\label{translation}
			En appliquant le foncteur de translation, on obtient des variantes du lemme \ref{bord}. Par exemple, soient $A\to B\xrightarrow{f} C\to A[1]$
			un triangle dans $D^{\geq 0}(X_{\et})$ et $j\geq 2$ un entier. Supposons 
			$\H^{j-1}(X^s,C^s)=0$ et que l'homomorphisme $\H^{j-2}(X^s,A^s)\to \H^{j-2}(X^s,B^s)$ est surjectif. Alors on a  $$f_*(F^1\H^{j}(X,B))\subset F^2\H^j(X,C),$$ et on a un carr\'e commutatif au signe pr\`es
			\[
			\begin{tikzcd}
				F^1\H^j(X,B) \arrow[r,"\varphi^1"] \arrow[d,"f_*"]   & H^1(k, \H^{j-1}(X^s,B^s)) \arrow[d,"\delta"]  \\
				F^2 \H^j(X,C) \arrow[r,"\varphi^2"] & H^2(k,\H^{j-2}(X^s,C^s)),
			\end{tikzcd}
			\]
			o\`u $\delta$ est le bord de la suite exacte longue de cohomologie associ\'ee \`a la suite exacte courte de $G$-modules
			\[0 \to \H^{j-2}(X^s,C^s)\to \H^{j-1}(X^s,A^s)\to \H^{j-1}(X^s,B^s) \to 0.\] 
		\end{rmk}

		\subsection{Le complexe de Lichtenbaum}
		Pour tout sch\'ema noeth\'erien $U$, on notera $\Gamma(U,2)$ le complexe de faisceaux sur le petit site \'etale de $U$ d\'efini par Lichtenbaum dans \cite[Definition 2.3]{lichtenbaum1987construction}. Ce complexe est concentr\'e en degr\'es $1$ et $2$; voir \cite[Remark 2.2]{lichtenbaum1987construction}. Si $U=\Spec(A)$ est affine, on \'ecrira $\Gamma(A,2)$ pour $\Gamma(U,2)$.
		
		On suppose d\'esormais que $\on{car}(k)=0$ et que la $k$-vari\'et\'e $X$ est lisse et g\'eom\'e\-triquement connexe. On pose $\Gamma(2)\coloneqq \Gamma(X,2)$.
		
		Soit $j\colon \Spec k(X)\hookrightarrow X$ l'inclusion du point g\'en\'erique. On a un isomorphisme canonique $\Gamma(k(X),2)\simeq j^*\Gamma(2)$. On note $\Gamma(k(X)/X,2)$ la fibre homotopique de l'homomorphisme canonique $\Gamma(2)\to Rj_*\Gamma(k(X),2)$ dans la cat\'egorie d\'eriv\'ee de $X_{\et}$ et par $\H^*(k(X)/k,\Gamma(2))$ l'hypercohomologie (\'etale) de $\Gamma(k(X)/X,2)$. On a donc un triangle exact
		\begin{equation}\label{generic-triangle}\Gamma(k(X)/X,2)\to\Gamma(2)\to Rj_*\Gamma(k(X),2)\to\Gamma(k(X)/X,2)[1]\end{equation}
		qui induit une suite exacte courte de $G$-modules	
		\begin{equation}\label{gamma-2-galois}
			0 \to \H^2(k_s(X),\Gamma(2))/\H^2(X^s,\Gamma(2)) \to \H^3(k_s(X)/X^s,\Gamma(2)) \to \H^3(X^s,\Gamma(2)) \to 0
		\end{equation}
		et une suite exacte longue
  \begin{equation}\label{generic-les}
\begin{split}
\cdots \to \H^i(k(X)/k,\Gamma(2))\to \H^i(X,\Gamma(2)) & \to \H^i(k(X),\Gamma(2)) \\
& \hspace{-0.3cm} \to \H^{i+1}(k(X)/k,\Gamma(2))\to \cdots
\end{split}
\end{equation}
		Dans (\ref{gamma-2-galois}), on a utilis\'e le fait que $\H^3(k_s(X),\Gamma(2))=0$; voir \cite[Proposition 4.3]{lichtenbaum1987construction}.

		On a le complexe de Gersten-Quillen
		\begin{equation}\label{gq-0}
			K_2(k_s(X))\to \bigoplus_{x\in (X^s)^{(1)}}k_s(x)^*\to \bigoplus_{x\in (X^s)^{(2)}}\Z.
		\end{equation}
		Posons
		\[Z\coloneqq \on{Ker}\left[\bigoplus_{x\in (X^s)^{(1)}}k_s(x)^*\to \bigoplus_{x\in (X^s)^{(2)}}\Z\right].\] 
		On obtient une suite exacte courte de $G$-modules
		\begin{equation}\label{gq-1}
			0 \to K_2(k_s(X))/H^0(X^s,\mc{K}_2)\to Z \to H^1(X^s,\mc{K}_2)\to 0.
		\end{equation}

		\begin{prop}\label{ct-raskind-lichtenbaum}
			On a un diagramme commutatif de $G$-modules
			\begin{equation*}\label{ct-raskind-licthenbaum-diag}
				\adjustbox{max width=\textwidth}{
					\begin{tikzcd}
						0 \arrow[r] & K_2(k_s(X))/H^0(X^s,\mc{K}_2)\arrow[r]\arrow[d,"\wr"] & Z \arrow[r]\arrow[d,"\wr"] & H^1(X^s,\mc{K}_2)\arrow[d,"\wr"]  \arrow[r] & 0	\\
						0\arrow[r] & \H^2(k_s(X),\Gamma(2))/\H^2(X^s,\Gamma(2))\arrow[r]  & \H^3(k_s(X)/X^s,\Gamma(2)) \arrow[r] &  \H^3(X^s,\Gamma(2))\arrow[r] &  0,
					\end{tikzcd}
				}
			\end{equation*}
			o\`u la suite exacte du haut est  (\ref{gq-1})  et celle du bas est induite par (\ref{gamma-2-galois}).
		\end{prop}
		
		\begin{proof}
			La d\'emonstration s'appuie sur des arguments contenus dans la preuve de \cite[Theorem 4.4]{lichtenbaum1990new}. Dans \cite[Theorem 4.4]{lichtenbaum1990new} Lichtenbaum travaille \`a la $2$-torsion pr\`es mais, comme on expliquera dans la suite, on peut utiliser des r\'esultats de Kahn \cite{kahn1996applications} pour supprimer cette hypoth\`ese. 
			Si $C$ est un complexe de faisceaux ab\'eliens \'etales sur $X$, on notera $\mc{H}^i(C)$  le 
			$i$-\`eme faisceau d'homologie de $C$.
			
			Comme le foncteur $j^*$ est exact, on a un isomorphisme canonique \[Rj^*\Gamma(2)\simeq j^*\Gamma(2)\simeq\Gamma(k(X),2).\] 
			L'adjonction entre $j^*$ et $j_*$ donne alors un homomorphisme canonique \[\Gamma(2)\to Rj_*\Gamma(k(X),2)\] qui fait partie de (\ref{generic-triangle}). Le complexe $\Gamma(2)$ \'etant acyclique en degr\'es $\neq 1,2$,	cette fl\`eche factorise par \[\psi\colon \Gamma(2)\to \tau_{\leq 3}Rj_*\Gamma(k(X),2),\]
			o\`u $\tau$ est le foncteur canonique de troncation. 
			Tous les triangles qui compl\`etent $\psi$ sont (non canoniquement) isomorphes entre eux. Donc, la suite (\ref{gamma-2-galois}) \'etant induite par  (\ref{generic-triangle}), elle est aussi induite par le triangle
			\begin{equation}\label{generic-triangle2}\Gamma(2)\xrightarrow{\psi} \tau_{\leq 3}Rj_*\Gamma(k(X),2)\to \on{Cone}(\psi)\to \Gamma(2)[1].
			\end{equation}
			Pour tout $x\in X$, soient $A_x$ l'henselis\'e strict de $X$ \`a $x$ et $K_x$ le corps des fractions de $A_x$. Pour tout $x\in X$, on a un diagramme commutatif 
			\begin{equation}\label{stalks}
				\begin{tikzcd}
					(j_x)^*\mc{H}^2(\Gamma(2))\arrow[d,hook,"\text{$(j_x)^*(\mc{H}^2(\psi))$}"]\arrow[r,"\simeq"]  & H^2(A_x,\Gamma(2))\arrow[d,hook]  & \arrow[l,"\simeq","\theta_{A_x}"'] K_2(A_x) \arrow[d,hook,"\psi_x"]\\
					(j_x)^*\mc{H}^2(Rj_*\Gamma(k(X),2)) \arrow[r,"\simeq"]  & H^2(K_x,\Gamma(2))  & \arrow[l,"\simeq","\theta_{K_x}"'] K_2(K_x).
				\end{tikzcd}
			\end{equation}
			(Noter que $H^2(A_x,\Gamma(2))=\mc{H}^2(\Gamma(A_x,2))$ et $H^2(K_x,\Gamma(2))=\mc{H}^2(\Gamma(K_x,2))$.)
			Les isomorphismes de gauche dans (\ref{stalks}) sont donn\'es par le fait que la construction  de $\Gamma(2)$ commute avec tout morphisme \'etale et par l'exactitude du foncteur $(j_x)^*$. Dans le carr\'e de droite dans (\ref{stalks}), l'isomorphisme $\theta_{A_x}$ est construit 
			dans \cite[Proposition 2.9]{lichtenbaum1990new}.
			
			Venons \`a la d\'efinition de $\theta_{K_x}$. De mani\`ere plus g\'en\'erale,
			pour tout corps $K$ on dispose d'un isomorphisme canonique
			\[\theta_K\colon K_2(K)\oi  \H^2(K,\Gamma(2)).\]
			La construction de $\theta_K$ est donn\'ee dans \cite[p. 195]{lichtenbaum1987construction}, et le fait qu'il est un isomorphisme suit de \cite[Theorem 4.5]{lichtenbaum1990new}. 
			L'homomorphisme $\psi_x$ appara\^{\i}t dans le complexe de Gersten-Quillen pour $A_x$ (on note
			$\kappa(P)$ le corps
			r\'esiduel en  le point $P$) :
			\begin{equation}\label{gersten-complex}
				0\to K_2(A_x)\to K_2(K_x)\to \bigoplus_{P \in \Spec(A_x)^{(1)}}K_1(\kappa(P))\to \bigoplus_{P\in \Spec(A_x)^{(2)}}K_0(\kappa(P))\to 0,	
			\end{equation}
			complexe qui est exact d'apr\`es la conjecture de Gersten pour $A_x$ d\'emontr\'ee par Quillen \cite[Theorem 5.11]{quillen}. Plus pr\'ecisement, le th\'eor\`eme de Quillen s'applique aux anneaux locaux des $k$-vari\'et\'es lisses, mais comme le complexe de Gersten-Quillen commute aux limites inductives d'anneaux avec homomorphismes de transition \'etales, la conjecture de Gersten pour $A_x$ s'en suit.
			
			Par l'adjonction entre $(j_x)^*$ et $(j_x)_*$, les $(\psi_x)_{x\in X^{(1)}}$ induisent un homomorphisme
			\begin{equation}\label{lichtenbaum-map}
				\phi=(\phi_x)_{x\in X^{(1)}}\colon\mc{H}^2(Rj_*\Gamma(k(X),2))\to \prod_{x\in X^{(1)}}(j_x)_*\G_m.
			\end{equation}
			Par construction des $\phi_x$, l'inclusion $k(X)\subset K_x$ induit un diagramme commutatif 
			\[
			\begin{tikzcd}
				k(x)^* \arrow[d,hook] & \arrow[l] K_2(k(X))\arrow[r,"\theta_{k(X)}","\simeq"'] \arrow[d] & H^2(k(X),\Gamma(2))\arrow[d] \arrow[ll,bend right=30,"\phi_x"]  \\
				k(x)_s^* & \arrow[l] K_2(K_x)\arrow[r,"\theta_{K_x}","\simeq"'] & H^2(K_x,\Gamma(2)), 
			\end{tikzcd}
			\]
			o\`u les homomorphismes de gauche sont les applications r\'esidu usuels qui apparaissent dans le complexe de Gersten-Quillen pour $X$ et $A_x$. Donc
			\begin{equation}\label{residu}
				K_2(k(X))\xrightarrow[\simeq]{\theta_{k(X)}} H^2(k(X),\Gamma(2))\xrightarrow{\phi_x} k(x)^* \text{ est l'application r\'esidu pour $x$.}
			\end{equation}
			Si $F$ est un faisceau \'etale sur $X$, pour toute section $s$ du faisceau \'etale $j_*j^*F$ il existe un ouvert dense $U\subset X$ tel que $s|_U$ provient de $F(U)$. Une application de cette observation \`a une r\'esolution injective de $Rj_*\Gamma(k(X),2)$ montre que pour toute section $s$ du faisceau \'etale $\mc{H}^2(Rj_*\Gamma(k(X),2))$, le nombre des points $x\in X^{(1)}$ tels que $\phi_x(s)\neq 0$ est fini. Donc $\phi$ s'ins\`ere dans la suite exacte de faisceaux \'etales
			\[0\to \mc{H}^2(\Gamma(2))\xrightarrow{\mc{H}^2(\psi)} \mc{H}^2(Rj_*\Gamma(k(X),2))\xrightarrow{\phi} \bigoplus_{x\in X^{(1)}}(j_x)_*\G_m\rightarrow \bigoplus_{x\in X^{(2)}}(j_x)_*\Z\to 0.\]
			Il existe alors un isomorphisme $\on{Coker}(\mc{H}^2(\psi))\simeq \mc{Z}$, o\`u $\mc{Z}$ est le faisceau \'etale
			\[\mc{Z}\coloneqq\on{Ker}\left[\bigoplus_{x\in X^{(1)}}(j_x)_*\G_m\to \bigoplus_{x\in X^{(2)}}(j_x)_*\Z\right].\]
			En particulier $p_*\mc{Z}= Z$. 
			
			En outre, l'homomorphisme $\mc{H}^1(\psi)$ est un isomorphisme. Ceci a \'et\'e d\'emontr\'e par Lichtenbaum \cite[p. 49]{lichtenbaum1990new} \`a la $2$-torsion pr\`es. L'\'enonc\'e complet suit du fait que, pour tout point $x\in X$, on a un diagramme commutatif (analogue \`a (\ref{stalks}))
			\[
			\begin{tikzcd}
				(j_x)^*\mc{H}^1(\Gamma(2))\arrow[d,"\text{$(j_x)^*(\mc{H}^1(\psi))$}"]\arrow[r,"\simeq"]  & H^1(A_x,\Gamma(2))\arrow[d,"\wr"]  & \arrow[l,"\simeq"'] K_3(A_x)_{\on{ind}} \arrow[d,"\wr"]\\
				(j_x)^*\mc{H}^1(Rj_*\Gamma(k(X),2)) \arrow[r,"\simeq"]  & H^1(K_x,\Gamma(2))  & \arrow[l,"\simeq"'] K_3(K_x)_{\on{ind}},
			\end{tikzcd}
			\]
			o\`u le carr\'e de gauche est induit par la compatibilit\'e de la construction de $\Gamma(2)$ avec tout morphisme \'etale et l'exactitude de $(j_x)^*$, et le carr\'e de droite est contenu dans \cite[bas de p. 399]{kahn1996applications}. On rappelle que, pour tout corps $K$, le groupe $K_3(K)_{\on{ind}}$ est d\'efini comme le conoyau de l'application naturelle $K_{3}^{\text{Milnor}}(K) \to K_3^{\text{Quillen}}(K)$.
			
			En conclusion, on a montr\'e que l'application $\mc{H}^1(\psi)$ est un isomorphisme, que $\mc{H}^2(\psi)$ est injective et que l'on a un isomorphisme de $G$-modules $\on{Coker}(\mc{H}^2(\psi))\simeq \mc{Z}$. Le lemme \cite[Lemma 4.3]{lichtenbaum1990new} donne alors un triangle exact
			\begin{equation}\label{generic-triangle3}\Gamma(2)\xrightarrow{\psi} \tau_{\leq 3}Rj_*\Gamma(k(X),2)\xrightarrow{\tilde{\phi}}\mc{Z}[-2]\to \Gamma(2)[1],
			\end{equation}
			o\`u $\mc{H}^2(\tilde{\phi})=\phi$. Les triangles (\ref{generic-triangle2}) et (\ref{generic-triangle3}) compl\`etent $\psi$, donc ils sont isomorphes. La suite de $G$-modules (\ref{gamma-2-galois}) \'etant induite par le triangle (\ref{generic-triangle2}), elle est alors isomorphe \`a la suite de $G$-modules induite par application de $Rp_*$ dans le triangle (\ref{generic-triangle3}):
			\[0\to \H^2(k_s(X),\Gamma(2))/\H^2(\Gamma(X^s,2))\to Z\to \H^3(X^s,\Gamma(2))\to 0,\]
			o\`u on a utilis\'e le fait que $p_*\mc{Z}= Z$.
			Consid\'erons le diagramme de suites exactes de $G$-modules suivant:
			\begin{equation}\label{ct-raskind-licthenbaum-reduced}
				\adjustbox{max width=\textwidth}{
					\begin{tikzcd}
						0\arrow[r] & K_2(k_s(X))/H^0(X^s,\mc{K}_2)\arrow[r] \arrow[d,"\wr"] & Z \arrow[r] \arrow[d,equal]& H^1(X^s,\mc{K}_2) \arrow[r] \arrow[d,"\wr"] &  0\\
						0 \arrow[r] & \H^2(k_s(X),\Gamma(2))/\H^2(\Gamma(X^s,2)) \arrow[r] & Z \arrow[r] & \H^3(X^s,\Gamma(2))  \arrow[r] & 0.
					\end{tikzcd}
				}
			\end{equation}
			Ici la suite du haut est (\ref{gq-1}), celle du bas est (\ref{generic-triangle}), l'isomorphisme vertical de gauche est induit par l'isomorphisme canonique $$\theta_{k_s(X)}\colon K_2(k_s(X))\oi  \H^2(k_s(X),\Gamma(2))$$ de \cite[Theorem 4.5]{lichtenbaum1990new} d\'ej\`a mentionn\'e ci-dessus. La commutativit\'e du carr\'e de gauche suit de l'assertion (\ref{residu}), consid\'er\'ee sur $k_s$. La fl\`eche verticale de droite est obtenue par passage aux quotients. 
			
			Le diagramme  (\ref{ct-raskind-licthenbaum-reduced}) est donc commutatif, ce qui \'etablit la Proposition 
			\ref{ct-raskind-lichtenbaum}.
		\end{proof}

		Notons
		$$CH^2(X)_{0} \coloneqq \on{Ker} [CH^2(X) \to CH^2(X^s)].$$
		Si $X$ satisfait (\ref{h4}), le groupe $H^0(X^s, \mc{K}_2)$  est  uniquement divisible (Th\'eor\`eme \ref{bloch-ct-raskind}(a)) et, comme montr\'e par le lemme \ref{phi-def} ci-dessous, l'application \[H^1(k,Z) \to H^1(k, H^1(X^s,\mc{K}_2))\] induite par  (\ref{gq-1}) s'identifie \`a la fl\`eche
		$ \Phi \colon   CH^2(X)_0 \to  H^1(k,S)$ dans le  th\'eor\`eme \ref{bloch-ct-raskind} ci-dessus.
		
		\begin{lemma}\label{phi-def}
			Supposons que $X$ satisfait (\ref{h4})	et soit $S$ le $k$-tore de groupe des cocaract\`eres $\Pic(X^s)$. Alors on a un carr\'e commutatif
			\[
			\begin{tikzcd}
				CH^2(X)_0\arrow[r,"\Phi"]\arrow[d,"\wr"] & H^1(k,S) \arrow[d,"\wr"] \\
				H^1(k, Z)\arrow[r] & H^1(k,H^1(X^s,\mc{K}_2)),
			\end{tikzcd}
			\]
			la fl\`eche horizontale inf\'erieure \'etant induite par la suite (\ref{gq-1}).
		\end{lemma}
		
		\begin{proof}
			La fl\`eche verticale de gauche provient de la preuve de \cite[Proposition 3.6]{CTR85}: elle est un isomorphisme pour toute $k$-vari\'et\'e lisse g\'eom\'etriquement connexe. L'isomorphisme vertical de droite vient du th\'eor\`eme \ref{bloch-ct-raskind}(b); ici on utilise le fait que $X$ satisfait (\ref{h4}). 
			
			La commutativit\'e du carr\'e est par construction de $\Phi$. En effet, soit $f$ le compos\'e \[f\colon CH^2(X)_0\oi H^1(k,Z)\to H^1(k,H^1(X^s,\mc{K}_2)).\] Alors $f$ est l'application qui appara\^{\i}t dans l'\'enonc\'e de \cite[Proposition 3.6]{CTR85}. Dans la preuve du th\'eor\`eme \ref{bloch-ct-raskind}(b) on a d\'efini $\Phi$ comme la compos\'ee de $f$ et de l'inverse de l'isomorphisme $H^1(k,S)\oi H^1(k,H^1(X^s,\mc{K}_2))$. Donc le carr\'e commute, comme voulu.
		\end{proof}

		\subsection{Application cycle motivique \'etale}
		Dans \cite[\S 5]{lichtenbaum1990new}, Lichtenbaum cons\-truit une application cycle \[\on{cl}_{\Gamma}\colon CH^2(X)\to \H^4(X,\Gamma(2)).\] Par \cite[Lemma 2.1]{kahn1996applications}, on a un isomorphisme $CH^2(X)\oi  \H^4(k(X)/X,\Gamma(2))$.
		Plus pr\'ecis\'ement, dans la preuve de \cite[Lemma 2.1]{kahn1996applications}, Kahn note que l'homomorphisme $\H^4(k(X)/X,\Gamma(2))\to \H^4(X,\Gamma(2))$ provenant de (\ref{generic-les}) est injectif, et que l'application $\on{cl}_\Gamma$ factorise par l'isomorphisme ci-dessus:
		\begin{equation}\label{top-left}
			\begin{tikzcd}
				CH^2(X)\arrow[r,"\simeq"] \arrow[dr,"\on{cl}_{\Gamma}"] & \H^4(k(X)/X,\Gamma(2)) \arrow[d,hook]  \\
				& \H^4(X,\Gamma(2)).
			\end{tikzcd}
		\end{equation}
		La fonctorialit\'e de la suite spectrale (\ref{hochschild-serre-hyper}) en $A$ donne un carr\'e commutatif
		\begin{equation}\label{top-right}
			\begin{tikzcd}
				F^1\H^4(k(X)/X,\Gamma(2))\arrow[r,"\simeq"] \arrow[d,hook]  & H^1(k,\H^3(k_s(X)/X^s,\Gamma(2))\arrow[d] \\
				F^1\H^4(X,\Gamma(2)) \arrow[r] & H^1(k,\H^3(X^s,\Gamma(2))).
			\end{tikzcd}
		\end{equation}
		Comme expliqu\'e au d\'ebut de \cite[p. 403]{kahn1996applications}, le fait que	la fl\`eche horizontale sup\'erieure est un isomorphisme suit du fait que l'on a  $\H^i(k(X)/X,\Gamma(2))=0$ pour $i\leq 2$. 
		
		Pour tout entier $n\geq 1$, on a un triangle exact
		\begin{equation}\label{gamma-triangle}\Gamma(2)\xrightarrow{\times n}\Gamma(2)\to \mu_n^{\otimes 2}\to \Gamma(2)[1]
		\end{equation}
		dans la cat\'egorie d\'eriv\'ee born\'ee de $X_{\et}$. Ce triangle appara\^{\i}t dans \cite[(12)]{kahn1996applications}. Il a \'et\'e construit dans \cite[Corollary 8.4]{lichtenbaum1987construction} sous l'hypoth\`ese $n$ impair. Cette restriction sur $n$ \'etait utilis\'ee dans le calcul de la torsion et de la cotorsion dans $K_{3,\on{ind}}$ de \cite[Lemma 8.2]{lichtenbaum1987construction}. Comme remarqu\'e par Kahn \cite[apr\`es (12)]{kahn1996applications}, la construction de Lichtenbaum de (\ref{gamma-triangle}) s'\'etend \`a tout $n\geq 1$ si on remplace \cite[Lemma 8.2]{lichtenbaum1987construction} par le th\'eor\`eme principal de \cite{kahn1992k3}. 
		
		D'apr\`es \cite[Proposition 5.4]{lichtenbaum1990new}, pour tout nombre premier $\ell$ on a un triangle commutatif
		\begin{equation}\label{bottom-left}
			\begin{tikzcd}
				CH^2(X)\arrow[r,"\on{cl}_{\Gamma}"] \arrow[dr,"\on{cl}_n"] & \H^4(X,\Gamma(2)) \arrow[d]  \\
				& H^4(X,\mu_{\ell^n}^{\otimes 2}), 
			\end{tikzcd}
		\end{equation}
		la fl\`eche verticale \'etant induite par (\ref{gamma-triangle}). 
		 
		Notons comme ci-dessus $CH^2(X)_{0} \coloneqq \on{Ker} [CH^2(X) \to CH^2(X^s)]$.
		Les diagrammes (\ref{top-left}) et (\ref{top-right}) donnent une application
		\[\Phi'\colon CH^2(X)_0\to H^1(k,\H^3(X^s,\Gamma(2))).\] La proposition \ref{ct-raskind-lichtenbaum} a des cons\'equences int\'eressantes 
		pour $\Phi'$. Ces cons\'equences ne seront pas utilis\'ees dans la preuve du th\'eor\`eme \ref{compatible}. Supposons que $X$ satisfait (\ref{h4})	et soit $S$ le $k$-tore de groupe des cocaract\`eres $\Pic(X^s)$. On dispose alors de l'application $\Phi\colon CH^2(X)_0\to H^1(k,S)$.

		\begin{prop}\label{image}
			Supposons que $X$ satisfait (\ref{h4})	et soit $S$ le $k$-tore de groupe des cocaract\`eres $\Pic(X^s)$. 
			
			(a) On a un isomorphisme $\on{Im}(\Phi')\simeq \on{Im}(\Phi)$.
			
			(b) Supposons que le corps $k$ est de type fini sur $\Q$ et que $X$ est une surface rationnelle. Alors $\on{Im}(\Phi')$ est finie.
		\end{prop}
		
		\begin{proof}
			(a) Par d\'efinition, l'application $\Phi'$ est la compos\'ee 
			\[CH^2(X)_0\oi  H^1(k,\H^3(k_s(X)/X^s,\Gamma(2)))\xrightarrow{\rho} H^1(k,\H^3(X^s,\Gamma(2))),\]
			la fl\`eche  $\rho$ \'etant induite par la suite (\ref{gamma-2-galois}), donc
			$\on{Im}(\Phi')= \on{Im}(\rho)$.
			Par la proposition \ref{ct-raskind-lichtenbaum} on sait que
			\[\on{Im}(\rho)\simeq \on{Im}[H^1(k, Z)\to H^1(k,H^1(X^s,\mc{K}_2))],\]
			la fl\`eche de droite \'etant induite par (\ref{gq-1}).
			Par le lemme \ref{phi-def}, on a
			\[\on{Im}[H^1(k, Z)\to H^1(k,H^1(X^s,\mc{K}_2))]\simeq \on{Im}(\Phi).\]
			On conclut que $\on{Im}(\Phi')\simeq \on{Im}(\Phi)$.
			
			(b) Sous les hypoth\`eses faites, \cite[Theorem 3(iv)]{CTS81} assure la finitude de $\on{Im}(\Phi)$. Toute surface rationnelle satisfait (\ref{h4}). La conclusion suit alors de la partie (a).
		\end{proof}

		\subsection{D\'emonstration du th\'eor\`eme \ref{compatible}}
		
		\begin{proof} 
			Par \cite[Theorem 1.1(iii)]{kahn1996applications}, on a un isomorphisme 
			$$H^0(X^s,\mc{K}_2)\oi \H^2(X^s,\Gamma(2)).$$
			D'apr\`es le th\'eor\`eme \ref{bloch-ct-raskind}(a),
			l'hypoth\`ese (\ref{h4}) donne que ce groupe 
			est uniquement divisible. Ainsi	l'homomorphisme 
			\[\H^2(X^s,\Gamma(2))\xrightarrow{\times \ell^n} \H^2(X^s,\Gamma(2))\] est un isomorphisme. L'hypoth\`ese (\ref{h4}) entra\^ine aussi $H^3(X^s,\mu_{\ell^n}^{\otimes 2})=0$. On peut alors appliquer la remarque \ref{translation} (cons\'equence de la proposition \ref{bord}) au triangle (\ref{gamma-triangle}) avec $j=4$. On obtient un carr\'e commutatif au signe pr\`es
			\begin{equation}\label{bottom-right}
				\begin{tikzcd}
					F^1\H^4(X,\Gamma(2)) \arrow[r] \arrow[d] & H^1(k,\H^3(X^s,\Gamma(2))) \arrow[d,"\delta_n"] \\
					F^2H^4(X,\mu_{\ell^n}^{\otimes 2})\arrow[r]  & H^2(k, H^2(X^s,\mu_{\ell^n}^{\otimes 2})),
				\end{tikzcd}
			\end{equation}
			o\`u $\delta_n$ est un homomorphisme de bord associ\'e \`a la suite de $G$-modules
			\begin{equation}\label{exact-because-divisible}
				0\to H^2(X^s,\mu_{\ell^n}^{\otimes 2})\to \H^3(X^s,\Gamma(2))\xrightarrow{\times \ell^{n}}\H^3(X^s,\Gamma(2))\to 0.
			\end{equation}

			Pour tout entier $n\geq 1$, les diagrammes commutatifs (\ref{top-left}), (\ref{top-right}), (\ref{bottom-left}) et (\ref{bottom-right}) nous donnent le diagramme commutatif
			\begin{equation}\label{big-diagram}
				\begin{tikzcd}
					CH^2(X)_0 \arrow[r,"\simeq"]\arrow[dr,bend right=15,"\on{cl}_{\Gamma}"]\arrow[ddr,bend right=30,"\on{cl}_n"]   & F^1\H^4(k(X)/X,\Gamma(2))\arrow[r,"\simeq"] \arrow[d,hook]  & H^1(k,\H^3(k_s(X)/X^s,\Gamma(2))\arrow[d,"f"] \\
					& F^1\H^4(X,\Gamma(2)) \arrow[r] \arrow[d] & H^1(k,\H^3(X^s,\Gamma(2))) \arrow[d,"\delta_n"] \\
					& F^2H^4(X,\mu_{\ell^n}^{\otimes 2})\arrow[r]  & H^2(k, H^2(X^s,\mu_{\ell^n}^{\otimes 2})),
				\end{tikzcd}
			\end{equation}
			o\`u le carr\'e du bas commute au signe pr\`es.
			Notons 
			\[\beta\colon CH^2(X)_0\oi H^1(k,\H^3(k_s(X)/X^s,\Gamma(2))\]
			le compos\'e des deux isomorphismes du haut dans (\ref{big-diagram}).
			D'apr\`es \cite[Theorem 1.1(iv)]{kahn1996applications} on a un isomorphisme de $G$-modules
			\[H^1(X^s,\mc{K}_2)\oi \H^3(X^s,\Gamma(2)).\]
			Sous l'hypoth\`ese (\ref{h4}), l'application naturelle $\on{Pic}(X^s)\otimes k_s^*\to H^1(X^s,\mc{K}_2)$ a noyau et conoyau uniquement divisibles		(Th\'eor\`eme \ref{bloch-ct-raskind}(b)). 
			On obtient un diagramme commutatif de suites exactes courtes de $G$-modules
			\begin{equation}\label{uniq-divisible}
				\begin{tikzcd}
					0 \arrow[r] & \on{Pic}(X^s)\otimes \mu_{\ell^n} \arrow[r] \arrow[d]& \on{Pic}(X^s)\otimes k_s^* \arrow[r,"\times \ell^n"]\arrow[d] & \on{Pic}(X^s)\otimes k_s^* \arrow[r]\arrow[d] & 0\\
					0\arrow[r] & H^2(X^s,\mu_{\ell^n}^{\otimes 2}) \arrow[r] & \H^3(X^s,\Gamma(2))\arrow[r,"\times \ell^{n}"] &\H^3(X^s,\Gamma(2))\arrow[r] & 0.
				\end{tikzcd}
			\end{equation}
			
			Ici, la fl\`eche verticale de gauche est induite par la commutativit\'e du carr\'e de droite.  Comme les deux  fl\`eches verticales de gauche sont \`a noyau et conoyau des $\Q$-vectoriels, et que les deux groupes de gauche sont annul\'es par $\ell^n$,
		le lemme du serpent donne  que la fl\`eche verticale de gauche est un isomorphisme. On n'a pas besoin de savoir que cette fl\`eche est induite par l'application cycle. 
			Par la proposition \ref{enoncereseaux}, il existe un entier $N\geq 1$ tel que l'homomorphisme de bord \[H^1(k,\on{Pic}(X^s)\otimes k_s^*)\set{\ell}\to H^2(k,\on{Pic}(X^s)\otimes \mu_{\ell^n})\] est injectif pour tout entier $n\geq N$. La commutativit\'e de (\ref{uniq-divisible}) entra\^ine alors que $\delta_n$ est injectif sur la torsion $\ell$-primaire pour tout entier $n\geq N$. Donc
	\begin{samepage}		
   \begin{equation}\label{noyau1}
				\on{Ker}(\Theta_n)\set{\ell}=\on{Ker}(\delta_n\circ f\circ\beta)\{\ell\}=\beta^{-1}(\on{Ker}(\delta_n\circ f)\set{\ell})=\beta^{-1}(\on{Ker}(f)\set{\ell})
			\end{equation}
        pour tout $n\geq N$.
	\end{samepage}

			On d\'eduit de la proposition \ref{MLH2}(a) que \[\on{Ker}(\Theta)\set{\ell}=\bigcap_{n\geq N} \on{Ker}(\Theta_n)\set{\ell},\] donc
			\begin{equation}\label{noyau2}
				\on{Ker}(\Theta)\set{\ell}=\beta^{-1}(\on{Ker}(f)\set{\ell})\simeq \on{Ker}(f)\set{\ell}.
			\end{equation} 
			Sous l'hypoth\`ese (\ref{h4}), on a
			\begin{align}\label{noyau3}
				\on{Ker}(f)\set{\ell}&\simeq \on{Ker}[H^1(k,Z)\to H^1(k,H^1(X^s,\mc{K}_2))]\set{\ell} \\
				&\simeq \on{Ker}[CH^2(X)_0\xrightarrow{\Phi} H^1(k,S)]\set{\ell}.\nonumber
			\end{align}
			Ici le premier isomorphisme provient de la proposition \ref{ct-raskind-lichtenbaum} et le deuxi\`eme du lemme \ref{phi-def}. En combinant (\ref{noyau2}) et (\ref{noyau3}), on conclut que $\on{Ker}(f)\set{\ell}\simeq \on{Ker}(\Phi)\set{\ell}$, comme voulu.
		\end{proof}
		
		\begin{rmk}
			Pour montrer que $\on{Ker}(\Theta)\set{\ell}= \on{Ker}(\Phi)\set{\ell}$ en tant que sous-groupes de $CH^2(X)_0$, il suffirait de montrer la commutativit\'e du carr\'e suivant:
			\[
			\begin{tikzcd}
				CH^2(X)_0 \arrow[r,"\simeq"]\arrow[d,"\wr"]  & H^1(k,Z) \arrow[d,"\wr"] \\
				F^1\H^4(k(X)/X,\Gamma(2)) \arrow[r,"\simeq"] & H^1(k,\H^3(k_s(X)/X^s,\Gamma(2))).
			\end{tikzcd}		
			\]
			Ici l'application du haut vient de \cite[Proposition 3.6]{CTR85}, celle du bas du carr\'e commutatif (\ref{top-right}), celle de gauche du carr\'e commutatif (\ref{top-left}) et celle de droite de la proposition \ref{ct-raskind-lichtenbaum}.
		\end{rmk}

		\section*{Remerciements}
		Le deuxi\`eme auteur remercie Fumiaki Suzuki et Burt Totaro pour plusieurs conversations utiles sur le sujet de cet article.


\begin{thebibliography}{CTSS83}
			
			\bibitem[AS22]{AS22}  Th.  Alexandrou et S. Schreieder, On Bloch's map for torsion cycles over non-closed fields. {\em Forum Math. Sigma} 11 (2023), Paper No. e53, 21 pp.
			
			\bibitem[Ara75]{arason1975cohomologische}
			J. K. Arason, Cohomologische Invarianten quadratischer Formen, {\em J. Algebra} {\bf 36} (1975), no. 3, 448--491.
			
			
			\bibitem[Blo79]{Blo79} S. Bloch,  
			Torsion algebraic cycles and a theorem of Roitman. {\em Compositio Math.}  {\bf 39} (1979), no. 1, 107--127. 
			
			\bibitem[Blo81]{Blo81} 
			S. Bloch, On the Chow groups of certain rational surfaces, {\em Ann. Sci. \'Ecole Norm. Sup.} (4) {\bf 14}(1981), no. 1, 41--59. 
			
			\bibitem[CT93]{CT93} 
			J.-L. Colliot-Th\'el\`ene,  Cycles alg\'ebriques de torsion et $K$-th\'eorie alg\'ebrique, in {\it Arithmetic Algebraic Geometry}, Trento 1991, Springer LNM {\bf 1553}.
			
			\bibitem[CTR85]{CTR85}  J.-L. Colliot-Th\'el\`ene et W. Raskind, $K_{2}$-cohomology and the second Chow group, {\em Math. Ann.} {\bf 270} (1985) 165--199.
			
			\bibitem[CTS81]{CTS81} 
			J.-L. Colliot-Th\'el\`ene et J.-J. Sansuc, On the Chow groups of certain rational surfaces: a sequel to a paper of S. Bloch,  {\em Duke Math. J.} {\bf 48} (1981), no. 2, 421--447.
			
			
			\bibitem[CTS83]{CTS83} 
			J.-L. Colliot-Th\'el\`ene et J.-J. Sansuc, Quelques gammes sur les formes quadratiques, {\em Journal of Algebra} 84 (1983) 449--467.
			
			\bibitem[CTSS83]{CTSS83} J.-L. Colliot-Th\'el\`ene, J.-J. Sansuc et C. Soul\'e,  Torsion dans le groupe de Chow de codimension deux, {\em Duke Math. J.} {\bf 50}, no. 3  (1983) 763--801.
			
			
			\bibitem[Del77]{deligne1977cohomologie}
			P. Deligne, {\em Cohomologie \'etale.} S\'eminaire de g\'eom\'etrie alg\'ebrique du Bois-Marie SGA $4\frac{1}{2}$. Lecture Notes in Mathematics, 569. Springer-Verlag, Berlin, 1977.
			
			\bibitem[Jan88]{J88} 
			U. Jannsen, Continuous \'etale cohomology, {\em Math. Ann.}  {\bf 280} (1988) 207--245.
			
			\bibitem[Kah93]{K93}
			B. Kahn, Descente galoisienne et $K_{2}$ des corps de nombres, {\em $K$-Theory} {\bf 7 } (1993), no. 1, 55--100.
			
			
			\bibitem[Kah96]{kahn1996applications} B. Kahn, Applications of weight-two motivic cohomology. \emph{Doc. Math.} {\bf 1} (1996), No. 17, 395--416.
			
			\bibitem[Kah92]{kahn1992k3}
			B. Kahn, $K_3$ d'un sch\'ema r\'egulier. \emph{C. R. Acad. Sci. Paris S\'er. I Math.} {\bf  315} (1992), no. 4, 433--436. 
			
			
			\bibitem[Kar91]{karpenko1991algebro}
			N. A. Karpenko,  Algebro-geometric invariants of quadratic forms. Traduction du russe de {\em Algebra i Analiz} {\bf  2} (1990), no. 1, 141--162 {\em Leningrad Math. J.} {\bf 2} (1991), no. 1, 119--138.
			
			\bibitem[KM13]{KM13}
			N. A. Karpenko et A. S. Merkurjev, On standard norm varieties. {\em Ann. Sci. \'Ec. Norm. Sup\'er.} (4) {\bf 46} (2013), no. 1, 175--214.
			
			
			
			\bibitem[Lic87]{lichtenbaum1987construction} 
			S. Lichtenbaum, The construction of weight-two arithmetic cohomology,  \emph{Inventiones math.} {\bf 88 } (1987), no. 1, 183--215.
			
			\bibitem[Lic90]{lichtenbaum1990new}
			S. Lichtenbaum,  New results on weight-two motivic cohomology, in {\it  The Grothendieck Festschrift}, Vol. III, 35--55, Progr. Math., {\bf 88}, Birkh\"auser Boston, Boston, MA, 1990.
			
			\bibitem[NSW08]{NSW} 
			J. Neukirch, A. Schmidt, K. Wingberg, {\it Cohomology of number fields}. Second edition. Grundlehren der mathematischen Wissenschaften [Fundamental Principles of Mathematical Sciences], {\bf 323}. {\em Springer-Verlag, Berlin}, 2008. xvi+825 pp.
			
			\bibitem[Noe70]{MN} M. Noether,
			\"{U}ber Fl\"{a}chen, welche Scharen rationaler Kurven besitzen (Habilitations-Schrift), {\em Math. Ann.} {\bf 3} (1870), S. 161--227.
			
			\bibitem[Qui72]{quillen} 
			D. Quillen, Higher algebraic $K$-theory. I. {\em Algebraic $K$-theory, I: Higher $K$-theories (Proc. Conf., Battelle Memorial Inst., Seattle, Wash., 1972)}, pp. 85--147. Lecture Notes in Math., Vol. {\bf 341}, Springer, Berlin 1973. 
			
			\bibitem[Sai91]{Sai91} S. Saito, Cycle map on torsion algebraic cycles of codimension two,
			{\em Inventiones math.} {\bf 106} (1991) 443--460.
			
			\bibitem[SS22]{SS22} 
			F. Scavia et  F. Suzuki, Noninjectivity of the cycle class map in continuous $\ell $-adic cohomology. {\em Forum Math. Sigma} 11 (2023), Paper No. e6, 19 pp.
			 
			
			\bibitem[Wei94]{weibel1994introduction}
			C. Weibel, {\em An introduction to homological algebra.} Cambridge Studies in Advanced Mathematics, {\bf 38}. Cambridge University Press, Cambridge, 1994.
			
		\end{thebibliography}
	\end{document}